\documentclass[11pt]{amsart}
\usepackage[foot]{amsaddr}
\usepackage[OT1,T1]{fontenc}
\usepackage{times}
\usepackage{calligra}
\usepackage{aurical}
\usepackage{mathpazo}
\usepackage{mathpple}
\usepackage{ulem}
\usepackage{amsthm}
\usepackage[mathscr]{euscript}

\usepackage{stmaryrd}
\usepackage{amssymb}
\usepackage{amsmath}
\usepackage{titletoc}
\usepackage{extarrows}
\usepackage{varioref}
\usepackage{bm}
\usepackage{amsthm}
\usepackage[all,cmtip]{xy}
\usepackage{tikz-cd}
\usepackage{color} 
\usepackage{soul}
\setulcolor{red}
\sethlcolor{yellow}
\setstcolor{blue}
\usepackage{amscd}
\usepackage{xfrac}    

\usepackage{amsfonts}
\usepackage[colorlinks,linkcolor=blue,citecolor=blue, pdfstartview=FitH]{hyperref}
\usepackage{graphicx}
\usepackage{geometry}
\usepackage{cite}

\numberwithin{equation}{section}

\theoremstyle{definition}
\newtheorem{thm}{Theorem}[section]
\newtheorem{theorem}[thm]{Theorem}

\newtheorem{lemma}[thm]{Lemma}
\newtheorem{corollary}[thm]{Corollary}
\newtheorem{proposition}[thm]{Proposition}

\newtheorem{remark}[thm]{Remark}

\newtheorem{definition}[thm]{Definition}

\newtheorem{assumption}[thm]{Assumption}

\newtheorem{example}[thm]{Example}

\newtheorem{defn-thm}[thm]{Definition-Theorem}

\newenvironment{observe}{\noindent\textcolor{blue}{\textit{Observation}}.}{\hfill \textcolor{blue}{$\blacktriangleleft$}\par}

\newcommand{\sE}{{\mathcal E}}
\newcommand{\sF}{{\mathcal F}}

\newcommand{\sH}{{\mathcal H}}
\newcommand{\sI}{{\mathcal I}}

\newcommand{\sL}{{\mathcal L}}
\newcommand{\sM}{{\mathcal M}}
\newcommand{\sN}{{\mathcal N}}
\newcommand{\sO}{{\mathcal O}}

\newcommand{\sT}{{\mathcal T}}

\newcommand{\sW}{{\mathcal W}}

\newcommand{\sY}{{\mathcal Y}}

\newcommand{\ssF}{{\mathscr F}}

\newcommand{\ssM}{{\mathscr M}}

\newcommand{\A}{\mathbb{A}}

\newcommand{\GG}{\mathbb{G}}

\newcommand{\PP}{{\mathbb P}}
\newcommand{\Q}{{\mathbb Q}}

\newcommand{\Z}{{\mathbb Z}}

\newcommand{\chr}{\operatorname{char}}

\newcommand{\Hom}{\operatorname{Hom}}

\renewcommand{\Im}{\operatorname{Im}}

\newcommand{\Pic}{\operatorname{Pic}}

\newcommand{\Spec}{\operatorname{Spec}}

\newcommand{\Proj}{\operatorname{Proj}}

\newcommand{\Higgs}{\operatorname{Higgs}}

\newcommand{\fppf}{{\operatorname{fppf}}}

\newcommand{\et}{{\operatorname{\acute{e}t}}}
\newcommand{\rk}{{\operatorname{rank}}}

\newcommand{\inj}{\hookrightarrow}

\newcommand{\surj}{\rightarrow\!\!\!\!\!\rightarrow}

\newcommand{\Td}{{\operatorname{Td}}}

\newcommand{\Sym}{{\operatorname{Sym}}}

\newcommand{\BP}{\mathbb{P}}

\newcommand{\btheorem}{\begin{theorem}}
	\newcommand{\etheorem}{\end{theorem}}
\newcommand{\bproposition}{\begin{proposition}}
	\newcommand{\eproposition}{\end{proposition}}
\newcommand{\bdefinition}{\begin{definition}}
	\newcommand{\edefinition}{\end{definition}}
\newcommand{\bcorollary}{\begin{corollary}}
	\newcommand{\ecorollary}{\end{corollary}}
\newcommand{\bproof}{\begin{proof}}
	\newcommand{\eproof}{\end{proof}}
\newcommand{\bremark}{\begin{remark}}
	\newcommand{\eremark}{\end{remark}}
\newcommand{\eexample}{\end{example}}
\newcommand{\bexample}{\begin{example}}
\newcommand{\elemma}{\end{lemma}}
\newcommand{\blemma}{\begin{lemma}}
\newcommand{\bobserve}{\begin{observe}}
	\newcommand{\eobserve}{\end{observe}}

\renewcommand{\phi}{\varphi}

\newcommand{\ee}{\end{eqnarray*}}
\newcommand{\be}{\begin{eqnarray*}}

\newcommand{\beq}{\begin{equation}}
	\newcommand{\eeq}{\end{equation}}

\newcommand{\bd}{\begin{enumerate}}
	\newcommand{\ed}{\end{enumerate}}
\newcommand{\bti}{\begin{tikzcd}}
	\newcommand{\eti}{\end{tikzcd}}

\renewcommand{\hat}{\widehat}

\def\pt{{\scriptscriptstyle\bullet}}

\renewcommand{\bf}[1]{\textbf{#1}}
\newcommand{\tx}[1]{\text{#1}}

\newcommand{\go}[1]{\mathfrak{#1}}

\begin{document}
	\title{A Noether-Lefschetz Theorem for Spectral Varieties with Applications}
		%\thanks{Grants or other notes
		%about the article that should go on the front page should be
		%placed here. General acknowledgments should be placed at the end of the article.}
	%\subtitle{Do you have a subtitle?\\ If so, write it here}
	
	%\titlerunning{Short form of title}        % if too long for running head
	\author{Xiaoyu Su$^1$} 
 % \thanks{Electronic address: %\texttt{suxiaoyu@mail.tsinghua.edu.cn}}}
\address{$^1$Yau Mathematical Science Center, Beijing, 100084, China.}
\email{suxiaoyu@mail.tsinghua.edu.cn}
	%\authorrunning{Short form of author list} % if too long for running head
		      
		%\emph{Present address:} of F. Author    if needed
    \author{Bin Wang$^2$ }
	 \address{$^2$Steklov Mathematical Institute of Russian Academy of Sciences, Moscow, 119991, Russia.
	}
	\email{binwang@mi-ras.ru}
%		\and 
%		Xueqing Wen \at
%		Yau Mathematical Science Center, Beijing, 100084, China. \\
%		\email{xueqingwen@mail.tsinghua.edu.cn}

	%\date{Received: date / Accepted: date}
	% The correct dates will be entered by the editor

	\maketitle
	\begin{abstract}
		We calculate the Picard group of generic (very general) spectral varieties living in the total space of a very ample line bundle over an algebraically closed field $k$ of odd characteristics or characteristic 0. We follow the strategy of Ravindra and Srinivas \cite{RS06,RS09} via formal Picard groups. As an application, we calculate the generic fibers of Hitchin systems over a smooth quintic surface of Picard number 1.
	\end{abstract}
	
	\section{Introduction}
	
	To study the Vafa-Witten theory built by Vafa and Witten in \cite{VW94}, Tanaka and Thomas in \cite{TT18,TT20} constructed the Vafa-Witten invariants on a polarised surface $(S,\sO_S(1))$  using the moduli space of stable $\omega_{S}$-valued Higgs sheaves. This moduli space is equipped with a Lagrangian fibration called the Hitchin map.
	
	To be more precise, let $k$ be an algebraically closed field with $\chr(k)=0$ or $\chr(k)=p\ge 3$ , $X$ be a smooth projective variety over $k$ with $\dim X\geq 2$ and $\sL$ be a very ample line bundle over $X$. A Higgs sheaf (with value in $\sL$) in this paper will be referred to as a pair $(E,\theta)$ where $E$ is a coherent sheaf on $X$, $\theta:E\rightarrow E\otimes\sL$ is an $\sO_{X}$-homomorphism. Similarly, as the curve case, the characteristic polynomial of $\theta$ defines a subvariety in $\Spec\Sym^\pt_{X}(\sL^\vee)$ (the total space of $\sL$) which is finite over $X$, and we call such subvarieties spectral varieties in $\Spec\Sym^\pt_{X}(\sL^\vee)$. By the classical Cayley-Hamilton theorem, $E$ can be treated as a coherent sheaf of generic rank 1 on the corresponding spectral variety. In the particular case that $X$ is a surface and $E$ is a vector bundle, $E$ is indeed a line bundle over the corresponding spectral surface, provided that the spectral surface is smooth. We show this in Proposition \ref{prop:bundle over spectral} by the theory of maximal Cohen-Macaulay modules on dimension two normal Noetherian local rings (cf. \cite{BBG97}, \cite{BD08}), (see Section 5). 
	
	Thus to study the generic fibers of the Hitchin map, we need to know more about Picard groups of generic smooth spectral varieties. Since we are in high dimensional case, the Picard groups, especially their connected component groups, are not as clear as the curve case. Our main result of this paper is to calculate the Picard groups of generic (very general) spectral varieties for the moduli of $\sL$-valued Higgs bundles (with certain ampleness) both in characteristic $0$ and characteristic $p\ge 3$. In particular,  we give a concrete description of generic fibers of such Hitchin systems over surfaces. We hope that our calculation of Picard groups can be used to study Vafa-Witten invariants constructed by Tanaka-Thomas \cite{TT18} via generic fibers of Hitchin systems.

% Thus it is important to study the Picard group of spectral varieties to understand the moduli of Higgs sheaves.
% 	\begin{assumption}\label{asm:key asm} We will make the following assumption
% 	\begin{enumerate}
% 	    \item If $\chr(k)=0$, we assume $\sL$ is big and globally generated.
% 	    \item If $\chr(k)>0$, we assume $\chr(k)>\dim(X)$ and $\#(k)$ is uncountable, $\sL$ is very ample and $X$ admits a lifting over $W_{2}(k)$.
% 	\end{enumerate}
% 	\end{assumption}
% 	\edit{The main purpose of the assumption in positive characteristics is to show the vanishing of certain coherent sheaves. Besides, since relative Picard scheme is not of finite type, instead it has countably many components, we also require our base field is uncountable.}
%	\[
%	f:X\rightarrow \PP^{N}
%	\]
%	is a generically finite morphism. We put $\sL=f^{*}\sO_{\PP^N}(1)$, then $\sL$ is a big and base-point free line bundle on $X$.
%	

	 We now introduce some notations. For a locally free coherent sheaf $\sE$ on $X$, we denote by $\bf V\sE:=\Spec_{X}\Sym_{\sO_X}^\pt\sE$ the associated vector bundle and by $\bf P\sE:= \Proj_X \Sym_{\sO_X}^\pt \sE\xrightarrow{\pi}X$ the associated projective bundle with relative $\sO_{\bf P\sE/X}(1)$ such that $\pi_*\sO_{\bf P\sE/X}(1)\cong \sE$. We can see that $\bf V\sE\subset \hat{\bf V\sE}:=\bf P(\sE\oplus\sO_{X})$ is an open subvariety and we call $D_{\infty}:=\bf P(\sE\oplus\sO_{X})\backslash\bf V\sE\cong\bf P(\sE)$ the infinite divisor. In the particular case that $\sE$ is the line bundle $\sL$, denoting $\bf P(\sL^{\vee}\oplus\sO_{X})$ by $Y$, it is straightforward that:
	\[
	H^{0}(Y,\pi^{*}\sL^{r}\otimes \sO_{Y/X}(r))=\oplus_{i=0}^{r}H^{0}(X,\sL^{i}).
	\]
	Thus there is a natural open immersion $$\Spec(\Sym_k^\pt\oplus_{i=1}^{r}H^{0}(X,\sL^{i})^{\vee})\hookrightarrow \bf P(	H^{0}(Y,\pi^{*}\sL^{r}\otimes \sO(r))^{\vee}).$$
	 \begin{definition}\label{def:spectral var}
	     Given $(X,\sL)$ as before, spectral varieties are divisors corresponds to closed points in  $\bm A:= \Spec(\Sym_k^\pt\oplus_{i=1}^{r}H^{0}(X,\sL^{i})^{\vee})$. And we call $\bm A$ the Hitchin base.
	 \end{definition}
	 
	 Here we can see that the spectral varieties are divisors in the linear system $|\pi^*\sL^r\otimes\sO_{Y/X}(r)|$ which do not intersect with $D_{\infty}$.  For a closed point $s\in\bm A$, we denote $X_{s}$ the corresponding effective divisor on $Y$ and we have a finite surjective map $\pi_{s}:X_{s}\rightarrow X$. We have the following proposition, see also Proposition \ref{prop:bigness of tauto}:
	\begin{proposition}
	    The line bundle $\pi^{*}\sL^{r}\otimes \sO_{Y/X}(r)$ is big and base point free over $Y$ for any $r>0$.
	\end{proposition}
	In the remaining part of this paper, we fix $r$ as the rank of Higgs bundles we shall consider in Section 5. To abbreviate the notations, we will use $\sW$ to denote the line bundle $\pi^{*}\sL^{r}\otimes \sO_{Y/X}(r)$ and $W$ to denote its global section $H^0(Y,\pi^{*}\sL^{r}\otimes \sO_{Y/X}(r))$.
	
	Our first goal of this paper is to calculate the Picard groups of generic (resp. very general when $\dim X=2$) spectral varieties, i.e. $X_{s}$ for $s$ in an open subvariety of $\bm A$ (resp. for a very general $s$ in $\bm A$). The Picard group of $Y$ can be given by the formula for projective bundles. However, because that $\pi^{*}\sL^{r}\otimes \sO_{Y/X}(r)$ is not ample on $Y$, we can not apply Noether-Lefschetz type theorem directly to calculate the Picard group of $X_{s}$. In characteristic 0, Ravindra and Srinivas in \cite{RS06,RS09} systematically deal with the Noether-Lefschetz type problem with the very ample line bundle replaced by a big and base point free line bundle which are nontrivial generalizations of the original Noether-Lefschetz type theorem. They modify Grothendieck's ideas in \cite{SGA2} and \cite[\S IV]{Har06} and take into consideration of the exceptional locus of the morphism defined by the corresponding big and base point free linear system. We modify their proofs, work over both characteristic 0 and odd characteristics and apply to spectral varieties. In positive characteristics, we always make the following assumption:
	\begin{assumption}\label{asm:main assumption}
	If $\chr(k)=p>0$, we assume that $X$ admits a $W_{2}(k)$ lifting and $p\geq 3$.
	\end{assumption}
	
% 	For the positive characteristic case, if $\dim(X)>2$ the Picard group of the generic spectral variety is given by adapting Grothendieck's Lefschetz theorem\sout{ which is also motivated  by \cite{RS06}}. For the $\dim X=2$ case, we use a different strategy i.e. Lefschetz pencils to calculate the Picard groups.

	Now let us state our theorem concerning the Picard group of a generic spectral variety when $\dim X\ge 3$.
	\begin{theorem}
	   If $\dim X\ge 3$, let $U$ be the open subvariety of the Hitchin base $\bm A$ parametrizing smooth spectral varieties, then under the assumption \ref{asm:main assumption}, for any closed point $s\in U$:
	   % \[
	   % 0\rightarrow A\rightarrow \Pic(Y)\rightarrow \Pic(X_{s})\rightarrow B\rightarrow 0
	   % \]
	   % here 
	   % \begin{enumerate}
	   %     \item $A:=\text{free abelian group generated by}\;\{E|E\;\text{is an irreducible effective divisor of}\; Y\text{ and }\dim g(E)=0 \}$
	   %     \item $B:=\text{abelian group generated by}\;\{E\cap X_{s}|E\;\text{is an irreducible effective divisor of}\; Y\text{ and }\dim g(E)=1 \}$
	   % \end{enumerate} 
	   %  In particular, we show that if $\sL$ is ample and globally generated, $A\cong\mathbb{Z}D_{\infty}$, $B=0$, thus, 
	     $\pi^{*}_s:\Pic(X)\rightarrow\Pic(X_{s})$ is an isomorphism. 
	\end{theorem}
% 	If $\dim X\geq 3$, by Proposition \ref{prop:bigness of tauto}, this follows from \cite[Theorem 2]{RS06}, and if $\dim X=2$ by Theorem \ref{thm:global gen}, this follows from \cite[Theorem 2]{RS09}.
	
	In characteristic 0, this is a special case of \cite{RS06}. And following Ravindra and Srinivas' strategy, since we consider quite special linear systems, we can also prove similar results in positive characteristics by a simple modification of Grothendieck's ideas in \cite{SGA2} and \cite[Chapter IV]{Har06}. 
	
	In the case that $\dim X=2$, things are more complicated. Thanks to a stronger cohomological result (see Theorem \ref{thm:global gen} and Corollary \ref{cor:surjection we need}) in our case, under the assumption that the relative Picard scheme $\underline{\Pic}^{0}_{X}$ is smooth, we prove the following analogous result:
	\begin{theorem}\label{thm:pre main isom}
		Let $X$ be a surface,  we assume $k$ is uncountable, $\underline{\Pic}^{0}_{X}$ is smooth, then there are very general members of $s$ in $\bm A$ such that $\pi^{*}_s:\Pic(X)\rightarrow\Pic(X_{s})$ is an isomorphism.
		
% 		$\sL$ is very ample and $X$ admits a lifting over $W_{2}(k)$. 
% 		\bd \item If $\dim X\geq 3$, and let $U$ be the open dense subvariety  of $\bm A$ parametrizing smooth spectral varieties then for any closed point $s\in U$
% 	    $\pi_{s}^{*}:\Pic(X)\rightarrow\Pic(X_{s})$ is an isomorphism.
	    
% 	    \item If $\dim X=2$, we assume  $\Pic^{0}(X)$ is smooth, \edit{then there are very general members of $s$ in $A$ such that the same results hold.}\ed
% 		\begin{itemize}
% 		    \item if $\dim X\geq 3$, $\pi_{s}^{*}:\Pic(X)\rightarrow\Pic(X_{s})$ is an isomorphism. In particular, the open subset $U$ can chosen as the locus in $\bm A$ parametrizing smooth spectral varieties;
% 		    \item if $\dim X=2$, $\pi_{s}^{*}:\Pic(X)\rightarrow\Pic(X_{s})$ is injective and the cokernel is a finitely generated $p$-torsion abelian group.
% 		\end{itemize}
	\end{theorem}
    For example, if $H^{1}(X,\sO_{X})=0$ by \cite[Corollary 9.5.13]{FGA} or if $H^{2}(X,\sO_{X})=0$ by \cite[Proposition 9.5.19]{FGA}, $\underline{\Pic}^{0}_{X}$ is smooth. For more examples, see \cite{Lie09}. As we mentioned in the beginning of the introduction, this theorem can be used to study generic fibers of Hitchin maps for moduli of $\sL$-valued Higgs bundles on $X$. In the last section, we give an application to Hitchin systems over a smooth quintic surface of Picard number one both in characteristic $0$ and odd characteristics, see Theorem \ref{thm:generic fibers are Hilbert schemes}. Roughly speaking, we show that
    \begin{theorem}
       Let $X$ be a smooth quintic surface of Picard number 1. We assume the rank of the torsion-free Higgs sheaves is greater than 3. Then generic fibers of Hitchin maps, if non-empty, are connected and isomorphic to Hilbert scheme of points of corresponding spectral surfaces. 
    \end{theorem}
    
    Let us now indicate how this relates to previous work. 
    
    It is a long-lasting question in algebraic geometry to calculate Picard groups of a divisor known as Noether-Lefschetz type theorem. Grothendieck defines the so-called (effective) Lefschetz condition in \cite{SGA2}  to solve Noether-Lefschetz problems systematically by formal geometry. This is also explained in detail by Hartshorne in \cite{Har70}. In our situation, the linear system considered is not ample but big and base-point free, so we can not apply Grothendieck-Lefschetz theorem directly. However, in characteristic 0, Grothendieck's idea are generalized by Ravindra and Srinivas to big and base-point free linear systems in \cite{RS06} when $\dim X\ge 3$ and in \cite{RS09} when $\dim X=2$. The assumption of characteristic 0 is essential because the authors use the resolution of singularities, exponential exact sequence for formal Picard groups and also when $\dim X=2$ the smoothness of Picard varieties is needed which naturally holds in characteristic 0. It is obvious that our calculation of Picard groups follows Ravindra and Srinivas's strategy. To be more precise, by our assumption of $W_2$-lifting in positive characteristics, when $\dim X\ge 3$, Ravindra and Srinivas' method can be applied directly, we here present a complete and simpler proof, since the linear system we consider is quite special though not very ample.  When $\dim X=2$, the exponential exact sequence for Picard groups can not hold for arbitrary thickening, we need to use a step by step lifting of line bundles to formal neighbourhood to verify Ravindra and Srinivas' ``\emph{Formal Noether-Lefschetz}" condition (see \cite[Definition 1]{RS09}). To make the step-by-step induction work, we need a stronger cohomological property than ``global-generation'' in \cite[Theorem 2]{RS09} which fortunately holds for the linear systems we consider here, see Theorem \ref{thm:global gen}. 
    
    We here must mention the very recent and deep results of Lena Ji \cite{Ji21} of Noether-Lefschetz type theory for normal threefolds in positive characteristics. In her thesis \cite{Ji21}, Ji uses a quite different method to calculate Picard groups of divisors lying in a linear system with sufficient ampleness on normal threefolds, see \cite[Corollary 3.4.2]{Ji21}. Moreover, she does not need to assume the existence of $W_2$ lifting and the smoothness of certain Picard variety. Our results here provided some special cases, i.e., various normal varieties other than $\BP^{3}$, that are not covered by Ji's \cite{Ji21}, see more in Remark \ref{rmk:comparison with Ji's result}.
    
    As we mentioned in the beginning of the introduction, via the nilpotent cones of Hitchin maps, Tanaka and Thomas \cite{TT18} can construct Vafa-Witten invariants in an algebro-geometric manner. Our original purpose was to study the generic fibers of corresponding Hitchin maps instead of nilpotent cones. By the classical BNR correspondence, Higgs bundles on base varieties can be treated as coherent sheaves on spectral varieties. By the theory of maximal Cohen-Macaulay modules in \cite{BBG97}, and the deep theory of maximal Cohen-Macaulay modules (of generic rank 1) over a  normal Noetherian local ring of dimension 2, the study of generic fibers turns into a study of Picard groups of spectral surfaces.

    We now close this section by briefly describing how the paper is organized. In the second section, we prove certain cohomological results which we will use later. In the third section, we prove our results on Picard groups when $\dim X\ge 3$. In the fourth section, we prove the results for very general spectral surfaces when $\dim X=2$. In the last section, we give an application to generic fibers of Hitchin systems over a smooth quintic surface of Picard number one.
    \vspace{10pt}
    
    \noindent \textbf{Acknowledgement}:
    The work starts during the second author's visit at Beijing International Center for Mathematical Research. The authors thank Qizheng Yin for his great support and very useful discussions during the visit. The work of Bin Wang was performed at the Steklov International Mathematical Center, Moscow, Russia and
supported by the Ministry of Science and Higher Education of the Russian Federation (agreement no.  075-15-2019-1614 ). The first author acknowledges support from Tsinghua Postdoctoral daily Foundation. And this manuscript was written during the first author's visit at the Institute for Advanced Study in Mathematics at Zhejiang University. We express our special thanks to the institute for its wonderful environment and support.

    \section{Vanishing Properties}
	In this section, we first start with some basic properties of projective bundles. 
% 	Following \cite[8.3.1]{EGAII},
	As in the introduction, we denote by $\hat{\bf V\sE}:= \bf P(\sE\oplus\sO_X)$ the associated projective compactification of $\sE$.
	
	The inclusion $\sE\subset \sE \cdot e\oplus\sO_X\cdot t$ induces the exact sequence of graded sheaves
	$$0\to(t)\to\Sym^\pt\left( \sE\cdot e\oplus\sO_X\cdot t\right)\rightarrow \Sym^\pt \sE\to 0,$$
	the isomorphism $\Sym^\pt \sE\cong \Sym^\pt( \sE\cdot\frac{e}{t})$,  
	and the surjective quotient map $\sE\oplus\sO_X\surj \sO_X$. 
	
	Then one has the open immersion $\bf V\sE\subset \bf P(\sE\oplus\sO_X)$ (locally given by $e\mapsto [e:1]$) with its complement  $\bf P \sE\subset \bf P(\sE\oplus\sO_X\cdot t)$, which is the zero locus of the degree $1$ homogeneous section $t$ in $\sO_{\bf P(\sE\oplus\sO_X)/X}(1)$.We call it the infinity divisor denoted by $D_{\infty}$. We also have the zero section $\sigma:X\to \bf P(\sE\oplus \sO_X)$ , which factors as the zero section of the open immersion $$\sigma: X\xrightarrow{\bm 0}\bf V\sE\subset \bf P(\sE\oplus\sO_X).$$ 
	Thus one has $\bf V\sE \sqcup D_\infty = \bf P(\sE\oplus\sO_X)$ and $\sigma(X)\subset\bf V\sE$ via the zero section. Besides there is a (line bundle) projection $p_{-\bm 0}:\hat{\bf V\sE}-\sigma(X)\to D_\infty$.
	
%	
%
%	\bproposition If $\sE$ has rank $1$, then $D_\infty\cong X$ and $p_{-\bm 0}$ is isomorphic to the bundle projection $\bf V(\sE^\vee)\to X$
%	\eproposition
%	
%	
%	
%	Then $Y=\bf V\sL^\vee\cup \bf V\sL$ with closed complements $D_\infty$ and $\sigma(X)$ respectively. Moreover, $\sO_{Y/X}(1)|_{\sigma(X)}=\sO_X, \sO_{Y/X}(1)|_{D_\infty}=\sL^\vee$ and $\pi^*\sL\otimes \sO_{Y}(1)\cong \sO_{\bf P(\sO_X\oplus\sL)}(1)$.
%	

Let us consider the case $\sE=\sL^\vee$. Then
$\pi:Y:=\bf P(\sL^\vee\oplus\sO_X)\to X$ and $\pi:D_\infty\cong X$. 
As we specify the line bundle $\sL$, we may denote $\sO_{\bf P(\sL^{\vee}\oplus \sO)/X}(1)$ by $\sO_{Y/X}(1)$ for simplicity. It is easy to see that $\sO_{Y/X}(1)|_{\bf V\sL^\vee}=\sO_{\bf V\sL^\vee}, \sO_{Y/X}(1)|_{D_\infty}=(\pi|_{D_{\infty}})^{*}\sL^\vee$.

%$\pi^*x=\bf P(\sL(x)^\vee\oplus k)\cong \PP^1\in \bm A_1(Y),$ $\pi^*C=\bf P(\sL|_C^\vee\oplus \sO_C)=\hat{\bf V\sL^\vee|_C}\in \bm A_2(Y)$. 
%
%----------------------------
%
%$\bm A^\pt (Y)=\pi^*\bm A^\pt(X)\oplus \pi^*\bm A^\pt(X)\cap D_\infty$ with $D_{\infty}^2=-\pi^*L\cap D_{\infty}=(-L\subset X=D_\infty)$, 
%
%$\zeta=c_1(\sO_{\bf P(\sE\oplus \sO)}(1))=D_\infty$.
%
%$\pi^*x\cap D_\infty=\{pt\}$, $\pi^*Z\cap D_{\infty}\cong D_{\infty}(\sL|_Z^\vee)\cong \bf P(\sL|_Z^\vee)\cong Z$.
%
%$\bm A^0 (Y)=\Z$, 
%
%$\bm A^1(Y)= \pi^*\bm A^1(X)\oplus \Z\cdot D_\infty=\{\hat{\bf V\sL|_D^\vee}\}\oplus \Z D_\infty$
%
%$\bm A_1(Y)= \pi^*\bm A_0(X) \oplus \pi^*\bm A_1(X)\cap D_\infty= \{\tx{fibers}\}\oplus \{C\subset X=D_\infty\}$.
%
%
%$N_1(Y)=\Z F\oplus \pi^*N_1(X)\cap D_{\infty}$
%
%%$\bm A^3=\{\tx{ points in }D_\infty=S\}$
%
%$K_Y=\pi^*(K_X-L)-2D_\infty$.
%
%------------------------------------

%------------------------------------
%
%$\sO_{\bf P((\sO\oplus\sL^\vee)\otimes\sL)}=\sO_{\bf P(\sO\oplus\sL^\vee)}\otimes\pi^*\sL=\sO_{Y/X}(1)\otimes\pi^*\sL$.
%
%$D_\infty(\sL)\sim D_{\infty}(\sL^\vee)+\pi^*L $.
%
%$D_\infty (\sL^\vee)=\sigma_{\bf V\sL}(X)$ which is the zero section of $\bf V\sL$.
%
%another way to compute $D_\infty\cap \sigma(X)=0$: 
%$$\sigma_{\bf V\sL^\vee}(X)\cap D_\infty(\sL^\vee)\sim (D_\infty+\pi^*L)\cap D_\infty=-L+L(\subset D_\infty)=0.$$

\begin{remark}
	 It is straightforward to see that $\pi^*\sL\otimes\sO_{Y/X}(1)|_{D_\infty}$ is trivial, and thus it is not ample. 
\end{remark}

%	
%	\begin{lemma}\label{lem:bigness under finite}
%		Let $\sM$ be a big line bundle over $X$ and $f:X'\rightarrow X$ is a finite map, then $f^{*}\sM$ is big on $X'$.
%	\end{lemma}
%    \begin{proof}
%    	By the generically globally generated criterion in \cite[Example 2.2.9]{Laz17}, we only need to prove that for any coherent $\sF$ on $X'$, there exists $m>0$ such that:
%    	\begin{equation}\label{eq:gen surj on cover}
%    	H^{0}(X', \sF\otimes f^{*}\sM^{m})\otimes\sO_{X'}\rightarrow \sF\otimes f^{*}\sM^{m}
%    	\end{equation}
%    	is generically surjective. Since $H^{0}(X',\sF\otimes f^{*}\sM^{m})\cong H^{0}(X,f_{*}\sF\otimes \sM^{m})$ and $f_{*}(\sF\otimes f^{*}\sM^{m})_{x}=\oplus_{y\in f^{-1}(x)}(\sF\otimes\sM^{m})_{y}$, then by the bigness of $\sM$ on X, 
%    	\begin{equation}\label{eq:gen surj on base}
%    	H^{0}(X,f_{*}\sF\otimes \sM^{m})\otimes\sO_{X}\rightarrow f_{*}\sF \otimes\sM^{m} 	
%    	\end{equation}
%    	is generically surjective, we know that $f^{*}\sM$ is big on $X'$. In particular, if surjectivity of \eqref{eq:gen surj on base} holds for $U$ in X, then the surjectivity of \eqref{eq:gen surj on cover} holds over $f^{-1}U$.
%    \end{proof}
    \begin{proposition}\label{prop:bigness of tauto}
    The line bundle $\pi^*\sL\otimes\sO_{Y/X}(1)$ is big and base-point free on $Y$. 
    \end{proposition}
	\begin{remark}
	This proposition still holds if we only assume $\sL$ is big and base point free. Thus in characteristic 0, when $\dim X\ge 3$,
%	by Theorem \ref{thm:global gen} and Corollary \ref{cor:surjection we need},
	 we can apply \cite[Theorem 2]{RS06} to the case that the line bundle $\sL$ is only assumed to be big and base point free. 
	\end{remark}
    
    \begin{proof}
    	
%    	\sout{	Let $\sF$ be a coherent sheaf on $Y$. By the bigness of $\sL$, there exists $n>0$ such that}
%    	$$	H^{0}(X,\pi_*(\sF\otimes\sO(n))\otimes\sL^{n})\otimes \sO_{X}\rightarrow \pi_*(\sF\otimes\sO(n))\otimes\sL^{n}
%    	$$
%    	\sout{is generically surjective. Consider the following diagram:}
%    	\begin{equation}
%    		\xymatrix{
%    			H^{0}(Y,\sF\otimes\pi^{*}\sL^{n}\otimes\sO(n))\otimes\sO_{Y}\ar[r]&\sF\otimes\pi^{*}\sL^{n}\otimes\sO(n)\\
%    			\pi^{*}(H^{0}(X,\pi_*(\sF\otimes\sO(n))\otimes\sL^{n})\otimes \sO_{X})\ar[r]\ar[u]^{\cong}&\pi^{*}\pi_*(\sF\otimes\sO(n))\otimes\pi^{*}\sL^{n}\ar[u]
%    		}
%    	\end{equation}
%    	\sout{by the relative ampleness of $\sO(1)$, we know that for sufficiently large $n$,  $\pi^{*}\pi_*(\sF\otimes\sO(n))\otimes\pi^{*}\sL^{n}\twoheadrightarrow (\sF\otimes\sO(n))\otimes\pi^{*}\sL^{n}$. Thus we prove that $H^{0}(Y,\sF\otimes\pi^{*}\sL^{n}\otimes\sO(n))\otimes\sO_{Y}\rightarrow\sF\otimes\pi^{*}\sL^{n}\otimes\sO(n)$ is generically surjective. Thus $\pi^{*}\sL\sO(1)$ is big over $Y$.}
%    	

    	Since $\sL$ is base-point free on $X$, and $\sL\otimes \pi_*\sO_{Y/X}(1)\cong \sO_X\oplus\sL$, we have the surjective evaluation map: 
    	\[e:H^0(X,\sO_X\oplus\sL)\otimes\sO_X\surj \sO_X\oplus\sL.\] Pullback it via $\pi^*$, we have the factorization of the evaluation map
    	\[	\bti H^{0}(Y,\pi^{*}\sL\otimes\sO_{Y/X}(1))\otimes_k\sO_{Y}\arrow[r,"{\tx{ev}_{\pi^*\sL\otimes\sO_{Y/X}(1)}}"]&\pi^{*}\sL\otimes\sO_{Y/X}(1)\\
    	\pi^*(H^{0}(X,\sL\otimes \pi_*\sO_{Y/X}(1))\otimes_k \sO_{X})\arrow[r,twoheadrightarrow,"{\pi^*e}"]\arrow[u,equal]&
    	\pi^*\sL	\otimes_{\sO_Y}		\pi^*\pi_*\sO_{Y/X}(1)\arrow[u,twoheadrightarrow]
    	\eti,\]
    	which shows that $\pi^{*}\sL\otimes\sO_{Y/X}(1)$ is base-point free. Hence $\pi^{*}\sL\otimes\sO_{Y/X}(1)$ is nef. To show that $\pi^{*}\sL\otimes\sO_{Y/X}(1)$ is big, we just have to check that the intersection number $(\pi^{*}\sL\otimes\sO_{Y/X}(1))^{\dim X+1}>0$.  But $\sigma (X)$ is a zero divisor of $\pi^{*}\sL\otimes\sO_{Y/X}(1)$, and $\pi^{*}\sL\otimes\sO_{Y/X}(1)|_{\sigma(X)}\cong \sL$. Then $(\pi^{*}\sL\otimes\sO_{Y/X}(1))^{\dim X+1}=(\sL)^{\dim X}>0$ which follows from the bigness of $\sL$ on $X$.
    	
    \end{proof}
    \begin{lemma}\label{lem:can bund formula}
    	Let $\omega_{Y}$ be the canonical line bundle of $Y$, then:
    	\[
    	\omega_{Y}\cong \pi^{*}(\omega_{X}\otimes\sL^{\vee})\otimes\sO_{Y/X}(-2)
    	\]
    \end{lemma}
    \begin{proof}
    	By the relative Euler exact sequence:
    	\[
    	0\to \Omega^1_{Y/X}\to \pi^*(\sL^\vee\oplus\sO_X)\otimes\sO_{Y/X}(-1)\to \sO_{Y}\to 0.
    	\]
    	we can calculate that $\omega_{Y}\cong \det(\Omega_{Y/X}^1)\otimes\pi^{*}\omega_{X}\cong \pi^{*}(\omega_{X}\otimes\sL^{\vee})\otimes\sO_{Y/X}(-2)$.
    \end{proof}
    
\begin{lemma}\label{lem: a vanishing theorem}
        $H^{i}(Y,\pi^{*}\sL^{-n}\otimes\sO_{Y/X}(-n))=0$ for $i<3$ and any $n\geq 1$.
    \end{lemma}
    \begin{proof}
        By Proposition \ref{prop:bigness of tauto}, this follows from  Kawamata-Viehweg Vanishing theorem in characteristic 0 for $i<\dim Y$. Now we prove it in positive characteristics. 
        
        % \sout{When $n=1$, we have $R\pi_{*}(\pi^{*}\sL^{-1}\otimes\sO_{Y/X}(-1))=0$, thus the vanishing holds.}
        
        % \sout{When $n>1$, $\pi_{*}(\pi^{*}\sL^{-n}\otimes\sO_{Y/X}(-n))=0$, thus $R\pi_{*}(\pi^{*}\sL^{-n}\otimes\sO_{Y/X}(-n))\cong R^{1}\pi_{*}(\pi^{*}\sL^{-n}\otimes\sO_{Y/X}(-n))$.
        % By the Grothendieck-Serre duality,} 
        %\begin{align*}
        %    R\pi_{*}(\pi^{*}\sL^{-n}\otimes\sO_{Y/X}(-n))&\cong R\pi_{*}R\sH om(\pi^{*}\sL^{n-1}\otimes\sO_{Y/X}(n-2)\otimes\pi^{*}\omega_{X},\omega_{Y})[-1]\\
         %   &\cong R\sH om(R\pi_{*}(\pi^{*}\sL^{n-1}\otimes\sO_{Y/X}(n-2)\otimes\pi^{*}\omega_{X}),\omega_{X})[-1]\\
          %  &=\sH om(\sL^{n-1}\otimes\pi_{*}\sO_{Y/X}(n-2),\sO_{X})[-1]\\
         %   &= \oplus_{\ell=1}^{n-1}\sL^{-\ell}[-1]
        %\end{align*}
    
        Recall $\pi^{*}\sL\otimes\sO_{Y/X}(1)\cong\sO_{\bf P(\sL\otimes\sO)/X}(1)$ and $Y=\bf P(\sL^\vee\oplus\sO)\cong \bf P(\sL\oplus \sO)$ is a projective bundle on $X$, thus by the direct image formula of projective bundles, $R\pi_{*}(\pi^{*}\sL^{-n}\otimes\sO_{Y/X}(-n))\cong \oplus_{\ell=1}^{n-1}\sL^{-\ell}[-1]$ for $n>1$ and $0$ for $n=1$.
        Thus we have $H^{i}(Y,\pi^{*}\sL^{-n}\otimes\sO_{Y/X}(-n))\cong H^{i-1}(X,\oplus_{\ell=1}^{n-1}\sL^{-\ell})$. Since $\sL$ is very ample, $X$ can be lift to $W_2(k)$ and $\chr(k)\ge 3$, then by \cite[Corollary 2.8,(2.8.2)]{DI87}, we have the vanishing $H^{i}(Y,\pi^{*}\sL^{-n}\otimes\sO_{Y/X}(-n))=0$ for $i<3$.
    \end{proof}

   By the construction as in \cite[Corollarie 8.8.4, and Theorem 8.9.1(crit\`ere de Grauert)]{EGAII}, 
%	one has commutative diagram of graded algebras 
%	\begin{footnotesize}
%		\[\bti \Sym_X^\pt a_X^*\Gamma(X,L)\arrow[r]\arrow[d,"open"]&\Sym_X^\pt\sL\arrow[d,"open"]\\
%		\Sym_X^\pt a_X^*(\Gamma(X,L)\oplus k) \arrow[r]\arrow[d,"\text{closed complement}"]&\Sym_X^\pt(\sL\oplus\sO_X)\arrow[d,"\text{closed complement}"]\\
%		\Sym_X^\pt a_X^*(\Gamma(X,L)) \arrow[r]&\Sym_X^\pt(\sL)
%		\eti\]
%	\end{footnotesize}
one has the induced open and closed decomposition:
\begin{footnotesize}
	\[\bti \bf V\sL \arrow[r]\arrow[d,hook,"\text{open}"]&\A_X^N\arrow[d,hook]\arrow[r]&\A^N\arrow[d]\\
	\bf P(\sL^{\vee}\oplus\sO_X)\cong Y\arrow[r]\arrow[rr,bend right,"{\phi_{|\pi^*\sL\otimes\sO_{Y/X}(1)|}}"']& \PP^N_X\arrow[r]&\PP_k^N\\
	X\arrow[u,"\sigma"]\arrow[rr,"{\phi_{|\sL|}}"']&&\PP^N_k\arrow[u,equal]
	\eti\]
\end{footnotesize}
    By the Grauert's criterion \cite[8.9.1]{EGAII}, $\sL$ is very ample, so $\bf V\sL$ to its image in $\A^N_k$ is the blowing down along $D_\infty$. Its closed complement is the closed immersion defined by $\phi_{|\sL|}$. Thus the projection $\phi_{|\pi^*\sL\otimes \sO_{Y/X}(1)|}$ is factored as the composition of a blowing down along $D_\infty$ and a closed immersion. 
    
%    \sout{Denote by $U$ the open subset $Y-D_\infty\cong \bf V\sL^\vee$ and $Z$ the image of $\phi_{|\pi^*\sL\otimes\sO(1)|}$. By Serre's criterion for normality, if $\dim(X)\geq 2$, then $Z$ is normal. There is a very ample divisor $H$ on $Z$ such that $\phi_{|\pi^*\sL\otimes\sO(1)|}^*\sO_Z(H)=\pi^*\sL\otimes\sO(1)$.}
    
    \begin{proposition}\label{prop:birat map}
	    The linear system $\left\vert\pi^{*}\sL^{r}\otimes \sO_{Y/X}(r)\right\vert=|\sW|$ induces a morphism
	    \[
	    g: Y\rightarrow \bf PH^{0}(Y,\pi^{*}\sL^{r}\otimes\sO_{Y/X}(r))=\bf P(W),
	    \]
	which is an immersion over $U$ and $g(D_{\infty})$ is a point. The image of $Y$ under $g$ is normal.
    \end{proposition}
    \begin{proof}
	    The previous arguments also hold for the linear system  $|\pi^*\sL^{\otimes r}\otimes\sO_{Y/X}(r)|$, denoting the projection by $g:Y\to \bf PH^{0}(Y,\pi^*\sL^{\otimes r}\otimes\sO_{Y/X}(r))$, one has $g$ factored as $g=v_r\circ\phi_{|\pi^*\sL\otimes\sO_{Y/X}(1)|}$, where $v_r$ is the $r$-fold Veronese embedding.
    
    	It is easy to see $g(Y)$ is integral. Since $\dim g(Y)\geq 3$ and has a unique isolated singularity, by Serre's criterion, $g(Y)$ is normal.
        
    \end{proof}
  
   	Denote by $Z$ the image of $g$.  We put $o=g(D_{\infty})$ which is the unique singularity of $Z$ and $Z^o=Z-\{o\}$ which is isomorphic to $\bf V\sL^{\vee}$. Since $Z$ is also isomorphic to the image of $\phi_{|\pi^*\sL\otimes\sO(1)|}$ via the $r$-fold Veronese embedding, there is a very ample line bundle $\sH$ on $Z$ such that $g^{*}\sH\cong \pi^{*}\sL\otimes \sO_{Y/X}(1)$, and $\sO_{\bf P(W)}(1)|_{Z}=\sH^{\otimes r}$. Recall that a coherent sheaf $\sF$ on $Z$ is $m$-regular with respect to $\sH$, if $H^{q}(Z,\sF\otimes\sH^{\otimes(m-q)}))=0$ for $q\geq 1$.
  	%  Now by the Theorem \ref{thm:main theorem in char 0}, we have:}
%  \begin{corollary}
%     Under the assumption that $\sL$ is very ample, the following exact sequence holds for generic $s$ \edit{(if $\dim X=2$, we need $s$ to be very general)}:
%     \begin{equation}\label{eq:key exact seq}
%     0\rightarrow \mathbb{Z}[D_{\infty}]\rightarrow\Pic(Y)\rightarrow\Pic(X_{s})\rightarrow 0
%     \end{equation}
%     %In particular, 
%   \end{corollary}

	\begin{theorem}\label{thm:global gen}
		Let $X$ be a smooth projective surface. Then under the Assumption \ref{asm:main assumption}, for $r>3$, we have $g_{*}\omega_{Y}\otimes\sO_{\bf P(W)}(1)=g_*(\omega_Y\otimes \sW)$ is Castelnuovo-Mumford $0$-regular with respect to the very ample line bundle $\sH$ on $Z$.
	\end{theorem}
    
    \begin{proof}
    	We first prove that $Rg_{*}\omega_{Y}\cong g_{*}\omega_{Y}$.
    % 	If $\chr(k)=0$, by the Grauert-Riemenschneider vanishing theorem (see \cite{GR70},\cite[Theorem 4.3.9]{Laz04} and \cite[Theorem 2.1]{Kol86I}), we have $g_*\omega_Y=Rg_*\omega_Y$.
    
 Recall that for $f : B\to C$ be a proper morphism between varieties and $\sF$ a coherent sheaf on $B$.
 The following are equivalent, for a proof see \cite{KM98} Proposition 2.69:
\begin{itemize}
    \item $H^q(B, \sF\otimes f^*\sO_C(H)) =0$ for $H$ sufficiently ample,
    \item $R^qf_*\sF=0$.
\end{itemize}
We take $B=Y, C=Z, \sF=\omega_{Y}$, and $H=\sH^{\otimes \ell} $ for $\ell\gg 0$. Then by Lemma \ref{lem: a vanishing theorem}, we have $R^{q}g_{*}\omega_{Y}=0$ for $q=1,2,3$. Since $\dim Y=3$, we have $Rg_*\omega_Y=g_*\omega_Y$.

% Then by Kawamata-Viehweg vanishing theorem (see \cite{Kawamata82},\cite{Viehweg82}, \cite[Theorem 4.3.1]{Laz04} and \cite[2.5]{KM98}), 

Since $r>3$, and $g^{*}\sH\cong \pi^{*}\sL\otimes\sO_{Y/X}(1)$, then again by Lemma \ref{lem: a vanishing theorem}, for $q=1,2,3$ we have 
    	\begin{small}
    	    \[ H^q(Z,g_*(\omega_Y\otimes \sW)\otimes \sH^{\otimes(-q)})
    		=H^{3-q}(Y,\pi^*\sL^{\otimes q-r}\otimes\sO_{Y/X}(q-r))=0.
    		\]
    	\end{small}
    And the zero-regularity of $g_*(\omega_Y\otimes_Y\sW)$ follows.
\end{proof}
\begin{remark}
If $\chr(k)=0$, then by Koll\'ar \cite[Theorem 2.1]{Kol86I}, $Rg_{*}\omega_{Y}=\omega_{Y}$ for a generic finite map between proper varieties with $X$ smooth .
\end{remark}
In particular, we have:
\begin{corollary}\label{cor:surjection we need}
 For any $\ell\ge 0$, we have the following surjection:
   \[H^{0}(Z,g_{*}(\omega_{Y}\otimes\sW))\otimes H^{0}(Z,\sO_Z(\ell))\twoheadrightarrow H^{0}(Z,g_{*}(\omega_{Y}\otimes \sW)\otimes \sO_Z(\ell)).\]
\end{corollary}
\begin{proof}
% By the Mumford's theorem,see \cite[Theorem 1.8.5]{Laz04}, and
% the $0$-regularity of $g_{*}\omega_{Y}\otimes \sO(1)$, we have：
By the $0$-regularity of $g_{*}(\omega_{Y}\otimes\sW)$ with respect to $\sH$, and the Mumford's theorem, see \cite[Chapter 5, Lemma 5.1]{FGA} or \cite[Theorem 1.8.5]{Laz04}, we have:
\[
H^{0}(Z,g_{*}(\omega_{Y}\otimes \sW))\otimes H^{0}(Z,\sH^{\ell})\rightarrow H^{0}(Z,g_{*}(\omega_{Y}\otimes \sW)\otimes\sH^{\ell})
\]
is surjective for any $\ell\ge 0$. Since $\sH^{r}=\sO_Z(1)$, we are done.
 \end{proof}

    \section{Higher Dimension Case}
   
%    Now let us recall the Grothendieck's Lefschetz conditions introduced in \cite[Definition 1]{RS06}, which is slightly weaker than that in\cite[Expos\'e X.2]{SGA2} and fits into our cases well. It is this weaker Lefschetz condition that helps explain why we can not have the results as in \cite[Theorem 3.1.8]{SGA2}.
%    \begin{definition}[\cite{RS06},\cite{SGA2} Expos\'e X.2 ]\label{def:lef cond}
%    	Let $T$ be a smooth projective variety, and $D$ an effective divisor in $T$. We put $\hat{T}$ the formal completion of $T$ along $D$. We say the pair $(T,D)$ satisfies the Lefschetz condition, denoted by $\text{Lef}^w(T,D)$ if for any any open neighborhood $V$ of $D$, and any locally free coherent sheaf $\sF$ on $V$, there is an open subset $V'\subset V$, such that the natural map: $H^{0}(V',\sF|_{V'})\rightarrow\ H^{0}(\hat{T},\hat\sF)$ is an isomorphism. Here $\hat{\sF}$ is the completion of $\sF$ along $D$.
%    	
%    	We say the pair $(T,D)$ satisfies the effective Lefschetz condition, which we denote by $\text{Leff}^W(T,D)$, if it satisfies $\text{Lef}^w(T,D)$ and in addition, for all locally free coherent sheaf $\ssF$ on $\hat{T}$, there exist an open neighbourhood $V$ of $D$ and a locally free coherent sheaf $\sF$ on $V$ such that $\hat{\sF}\cong \ssF$.
%    \end{definition}
 In this section, we consider the spectral variety $X_s$ for $s\in \bm A$. First, we check the smoothness of a generic spectral variety. This can be done by considering spectral varieties defined by equations:
    \[
    \lambda^{r}+a_{r}=0,
    \]
    where $a_{r}\in H^{0}(X,\sL^{\otimes r})\subset \bm A$. By the vary ampleness of $\sL$, for generic $a_{r}$, the zero divisor of $a_{r}$ is smooth. Then the corresponding spectral variety, as a $r$-cyclic cover of $X$ ramified over $\text{zero}(a_{r})$, is smooth. Thus we can see that generic spectral varieties are smooth.
    
 Let $X_{s}$ be a generic smooth spectral variety with $s\in \bm A$, our goal in this section is to show that the natural map $\pi^{*}:\Pic(X)\rightarrow \Pic(X_{s})$ is an isomorphism when $\dim X\ge 3$. Considering the following exact sequence:
 \[
 0\rightarrow \mathbb{Z}[D_{\infty}]\rightarrow\Pic(Y)\rightarrow\Pic(U)\cong\pi^{*}\Pic(X)\rightarrow 0,
 \]
 we only need to prove the following exact sequence for generic $s$:
 \begin{equation}\label{eq:key exact seq}
 	0\rightarrow \mathbb{Z}[D_{\infty}]\rightarrow\Pic(Y)\rightarrow\Pic(X_{s})\rightarrow 0.
 \end{equation}
 Since we also consider positive characteristics and also for the completeness of the paper, we simplified and adapted the proofs in \cite{RS06} to our cases. 
 
  Let us denote the formal completion of $Y$ along $X_{s}$ by $\hat{Y}_{s}$  and the $\ell$-th thickening of $X_s$ by $X_{s,\ell}$. Then $\hat Y_s=\displaystyle\lim_{\longrightarrow}X_{s,\ell}$ in the category of locally ringed spaces. We denote the defining ideal $X_{s}$ by $\sI_s\cong\sW^{-1}$. One has the exact sequence:
    \[0\to \sI_{s}^{\ell}\to \sO_Y\to \sO_{X_{s,\ell}}\to 0.\] 
     
    \bproposition \label{prop:inj of formal pic}If $\dim X\geq 3$, $\Pic(\hat Y_s)\cong \Pic(X_s)$. If $\dim X=2$, the natural map $\Pic(\hat{Y}_{s})\rightarrow \Pic(X_{s})$ is an injection.
    \eproposition
    \begin{proof}
        
        One has the following exact sequence:
        \[
        0\rightarrow\sI^{m}_s/\sI_s^{m+1}\rightarrow\sO^{\times}_{X_{s,m+1}}\rightarrow \sO^{\times}_{X_{s,m}}\rightarrow 0
        \]
        
        Since $\pi^*\sL^{-m}\otimes \sO_{Y/X}(-m)|_{X_s}\cong \pi_{s}^*\sL^{-m}$, we have $$H^i(X_s,\sfrac{\sI^m}{\sI^{m+1}})\cong H^i(X_s,\pi^*\sL^{-m}\otimes \sO_{Y/X}(-m)|_{X_s})$$ If $\dim X\geq 3$, by the $W_2(k)$-Kodaira Vanishing theorem in positive characteristics and the Kawamata-Viehweg Vanishing theorem for big and base point free line bundle in charateristic 0, $H^i(X_s,\sfrac{\sI^m}{\sI^{m+1}})=0$ for $i=1,2$ and $m\geq 1$. Then by \cite[Expos\'e XI.1]{SGA2}, We get the isomorphism. If $\dim X=2$, similarly, we get $H^i(X_s,\sfrac{\sI^m}{\sI^{m+1}})=0$ for $i=1$, thus $\Pic(\hat{Y}_{s})\rightarrow \Pic(X_{s})$ is an injection.
    \end{proof} 

Now let us recall the modified Grothendieck's Lefschetz conditions introduced in \cite[Definition 1]{RS06}, which is weaker than that in\cite[Expos\'e X.2]{SGA2} and fits into our cases well. It is this weaker Lefschetz condition that helps explain why we can not have the results as in \cite[Theorem 3.1.8]{SGA2}.
    \begin{definition}[\cite{RS06},\cite{SGA2} Expos\'e X.2 ]\label{def:lef cond}
    	Let $T$ be a smooth projective variety, and $D$ an effective divisor in $T$. We put $\hat{T}$ the formal completion of $T$ along $D$. We say the pair $(T,D)$ satisfies the weak Lefschetz condition, denoted by $\text{Lef}^w(T,D)$ if for any any open neighborhood $V$ of $D$, and any locally free coherent sheaf $\sF$ on $V$, there is an open subset $V'\subset V$, such that the natural map: $H^{0}(V',\sF|_{V'})\rightarrow\ H^{0}(\hat{T},\hat\sF)$ is an isomorphism. Here $\hat{\sF}$ is the completion of $\sF$ along $D$.
    	
    	We say the pair $(T,D)$ satisfies the weak effective Lefschetz condition, which we denote by $\text{Leff}^w(T,D)$, if it satisfies $\text{Lef}^w(T,D)$ and in addition, for all locally free coherent sheaf $\ssF$ on $\hat{T}$, there exist an open neighbourhood $V$ of $D$ and a locally free coherent sheaf $\sF$ on $V$ such that $\hat{\sF}\cong \ssF$.
    \end{definition}

  Notice that for any open neighborhood of $X_{s}$ and any locally free coherent sheaf $\sF$ on $V$, it can be extended to a reflexive sheaf on $Y$. %(by cf. \cite[Corollaire (9.4.8)]{EGAI}, $\sF$ extend to a coherent sheaf $\tilde{\sF}$ on $Y$ and then take double dual} 
    The following proposition shows that the pair $(Y,X_{s})$ satisfies the weak Lefschetz condition in Definition \ref{def:lef cond}:
    \begin{proposition}\label{prop:mod lef cond}
    	For any reflexive sheaf $\sN$ on $Y$ which is locally free in an open neighborhood of $X_s$, there exists an integer $m$ such that the natural map:
    	\[
    	H^{0}(Y,\sN(mD_{\infty}))\rightarrow H^{0}(\hat{Y}_{s},\hat{\sN})
    	\]
    	is an isomorphism.
    \end{proposition}
    \begin{proof}
    	Since $X_{s}\cap D_{\infty}=\emptyset$, $H^{0}(\hat{Y}_{s},\hat{\sN})\cong H^{0}(\hat{Y}_{s},\hat{\sN(\ell D_{\infty})})$ for any $\ell\in\mathbb{Z}$.
    % 	By the Serre's duality and formal duality, we only need to show for some $m$:
    % 	\[
    % 	H^{n}_{X_{s}}(Y,\sN^{\vee}\otimes (-mD_{\infty})\otimes\omega_Y)\rightarrow H^{n}(Y,\sN(-mD_{\infty})\otimes\omega_Y)
    % 	\]
    % 	is an isomorphism.
    	
    % 	Consider the following exact sequence:
    	
    % 	\[
    % 	\xymatrix{
    % 	H^{n-1}(Y\backslash X_{s}, \sN^{\vee} (-mD_{\infty})\otimes\omega_Y)\ar[d]\ar[r]& 	H^{n}_{X_{s}}(Y,\sN^{\vee} (-mD_{\infty})\otimes\omega_Y)\ar[d]\ar[r]&H^{n}(Y,\sN^{\vee}(-mD_{\infty})\otimes\omega_Y)\ar[d]\ar[r]&H^{n}(Y\backslash X_{s}, \sN^{\vee}(-mD_{\infty})\otimes\omega_Y)\\
    % 	H^{n-1}(Y\backslash X_{s}, \sN^{\vee}\otimes\omega_Y)\ar[d]\ar[r]& 	H^{n}_{X_{s}}(Y,\sN^{\vee}\otimes\omega_Y)\ar[d]\ar[r]&H^{n}(Y,\sN\otimes\omega_Y)\ar[d]\ar[r]&H^{n}(Y\backslash X_{s}, \sN^{\vee}\otimes\omega_Y)
    % 	}
    % 	\]
 Recall that, as a divisor, $X_s\simeq r(D_\infty+\pi^*L)$ and $X_s-iD_\infty$ is ample for $0<i<r$. In fact, $D_\infty+(1+\epsilon)\pi^*L$ is ample for any $\epsilon>0$, i.e., $r(D_\infty+\pi^*L)+m\pi^*L$ is ample for all $m>0$. This is because $\sO_Y(r(D_\infty+\pi^*L)+m\pi^*L)\cong \sO_{\bf P(\sL^m\oplus\sL^{r+m})/X}(1)$
and $\sL^m\oplus\sL^{r+m}$ is an ample vector bundle. 

The map in this proposition is induced by first considering the exact sequence for each thickening $X_{s,n}$ (the sequence is exact because we assume $\sN$ is locally free along $X_s$)
    	\[\xymatrix{ 0\ar[r] &\sN(mD_\infty-nX_s)\ar[r]& \sN(mD_\infty)\ar[r]^{t_n\ \ \ \ \ \ \ \ \ }&\sN(mD_\infty)|_{X_{s,n}}\cong \sN|_{X_{s,n}}\ar[r]& 0
    		},
    		\]
    then taking the inverse limit $\displaystyle\lim_{\leftarrow}H^0(t_n)$ (cf. \cite[Chapter 8, Corollary 8.2.4]{FGA}). To prove the proposition, we have to show that for $n$ sufficiently large $H^0(t_n)$ is both injective and surjective (cf. \cite[8.2.5.2]{FGA} or \cite[Ch. 0, 13]{EGAIII-1}). This can be deduced from the vanishing of the cohomologies
    \[\begin{split} H^0(Y,\sN(mD_\infty-nX_s))&=H^{\dim Y}(Y,\omega_Y\otimes \sN^\vee(nX_s-mD_\infty))^\vee=0,\\
    H^1(Y,\sN(mD_\infty-nX_s))&=H^{\dim Y-1}(Y,\omega_Y\otimes \sN^\vee(nX_s-mD_\infty))^\vee=0.
    \end{split}\]
    This is because $nX_s-mD_\infty=m(X_s-D_\infty)+(n-m)X_s$, $(X_s-D_\infty)$ is ample and $X_s$ is nef. Then by the Fujita vanishing theorem \cite[Theorem (1)]{Fuj83} (see also \cite[Remark 1.4.36]{Laz04}) for $m$ sufficiently large, and all $n>m$, we have the desired vanishing of cohomologies, which complete the proof.
    \end{proof}
%     Then we have
% \bcorollary\label{pair Lefschetz} The pair $(Y,X_s)$ satisfies the Lefschetz condition $\tx{Lef}(Y,X_s)$.
% \ecorollary
    \begin{proposition}\label{prop:ker of pic res}
    We have the following exact sequence:
    \[
    0\rightarrow \mathbb{Z}D_{\infty}\rightarrow\Pic(Y)\rightarrow \Pic(\hat{Y}_{s}).
    \]
    Since $\Pic(\hat{Y}_{s})\rightarrow\Pic(X_{s})$ is injective, we also have:
    \[
    0\rightarrow \mathbb{Z}D_{\infty}\rightarrow\Pic(Y)\rightarrow \Pic(X_{s}).
    \]
    \end{proposition}
    \begin{proof}
        It is obvious that $D_{\infty}$ is trivial when restricts to $\hat{Y}_{s}$.
        % \sout{If $\sM$ is a line bundle on $Y$, and $\sM|_{X_{s}}$ is trivial, then $\hat{\sM}$ is also trivial, because $\Pic(\hat{Y}_{s})\hookrightarrow\Pic(X_{s})$.}\footnote{This has to be proved later. The key point is the vanishing of $H^{i}(Y,\sL^{-n}\otimes\sO(-n))$ for $ i< \dim Y$.} 
        Let $Z^o=Y-D_\infty$, the exact sequence is deduced if we can show the injectivity of $\Pic(Z^o)\to \Pic(X_s)$. In other words,  for line bundle $\sM$ on $Z^o$ such that  $\sM|_{X_{s}}$ is trivial, we have to show $\sM$ is trivial. By the Proposition \ref{prop:inj of formal pic}, $\Pic(\hat{Y}_{s})\hookrightarrow\Pic(X_{s})$, we know that $\hat{\sM}$ is trivial on $\hat Y_s$.
        
        Then there is an invertible section of $\hat{\sM}$. By the Proposition \ref{prop:mod lef cond}, there exists an open neighborhood $V$ such that the isomorphism
        \[
        H^{0}(Y,\sM(mD_{\infty}))\rightarrow  H^{0}(\hat{Y}_s,\hat{\sM})
        \]
        factors through $H^{0}(V, \sM(mD_{\infty})|_V)\rightarrow H^{0}(\hat{Y}_s,\hat{\sM})$ which is also an isomorphism (it is injective because of the torsion freeness). Thus $\sM(mD_{\infty})$ has an invertible section in $U$ which means $\sM(mD_{\infty})\cong \sO_{Y}(\ell D_{\infty})$ for some $\ell$. We finish the proof.
        \end{proof}
    % In the following, we introduce Grothendieck's effective Lefschetz condition for a pair $(Z,D)$ as in the Definition \ref{def:lef cond}.
    % \begin{definition}[\cite{SGA2} Expos\'e X.2 ]
    %     We say the pair $(Z,D)$ satisfies the effective Lefschetz condition, denoted as $\tx{Leff}(Z,D)$ if \edit{it satisfies the Lefschetz condition and} for any locally free formal sheaf $\ssF$ on $\hat{Z}$, there exists a locally free sheaf $\sF$ in an open neighbourhood $U$ of $D$ such that $\hat{\sF}\cong\ssF$.
    % \end{definition}
    \begin{proposition}
        For the pair $(Y,X_{s})$, $\tx{Leff}^w(Y,X_{s})$ holds.
    \end{proposition}
    \begin{proof}
        % Since $\sM:=\pi^{*}\sL^{r}\otimes\sO_{Y/X}(r)$ is restricted to an ample invertible sheaf on $X_{s}$, 
        By \cite[Chapter IV, Theorem 1.5]{Har70} (also see \cite[5.2.4]{EGAIII-1} and \cite[8.4.3]{FGA}), since $\sW=\pi^{*}\sL^{r}\otimes\sO_{Y/X}(r)$ restricts to a very ample line bundle in an open neighbourhood of $X_{s}$, for any locally free coherent formal sheaf $\ssF$ on $\hat{Y}_s$, we have the exact sequence:
        \[\sO_{\hat Y_s}(-m_1)^{\oplus M_1}\xrightarrow{\hat \phi}\sO_{\hat Y_s}(-m_2)^{\oplus M_2}\surj \ssF\to 0.\]
        For notation ease, we simply write $\hat{\sW^{m}|_{Y_{s}}}$ by $\sO_{\hat Y_s}(m)$.
        
        By Corollary \ref{prop:mod lef cond}, $\tx{Lef}^w(Y,X_s)$ holds and $$\hat \phi\in \sH om(\sO_{\hat Y_s}(-m_1)^{\oplus M_1},\sO_{\hat Y_s}(-m_2)^{\oplus M_2})\cong \hat{\sH om}(\sO_{Y}(-m_1)^{\oplus M_1},\sO_{Y}(-m_2)^{\oplus M_2})$$ is algebrizable by $\phi\in\Gamma(Y, \sH om_{Y}(\sO(-m_1)^{\oplus M_1},\sO(-m_2)^{\oplus M_2}\otimes\sO_Y(m_3 D_\infty)))$. Then we have $\hat{\tx{CoKer}(\phi)}\cong \ssF$. For any $y\in X_s$, we have $\sI_{s}\subset \go m_y$. Then completion along $\sI_{s}$ and then completion along ${\hat{\go m_y}}^{\sI_s}$, is equal to directly complete at $\go m_y$. 
        
        This means ${\hat{\tx{Coker}(\phi)}}^{\go m_y}$ is locally free at each point $y\in X_s$. By faithful flatnees of the completion along a maximal ideal, $\tx{Coker}(\phi)$ is locally free after being localized at each closed point $y\in X_s$. Thus  $\tx{Coker}(\phi)$ is locally free over a neighbourhood $U$ of $X_s$.
       
        % \[
        % \hat M^{-M}\rightarrow \hat{\sF}\rightarrow 0
        % \]   
        % similarly, for $\hat{\sF}^{\vee}$, we also have:
        % \[
        % \hat{\sM}^{-N}\rightarrow \hat{\sF}^{\vee}\rightarrow 0
        % \]
        % then take the dual exact sequence:
        % \[
        % 0\rightarrow \hat{\sF}\rightarrow \hat{\sM}^{N}
        % \]
        % now we have:
        % \[
        % \hat{\phi}:\hat{\sM}^{-M}\rightarrow \hat{\sM}^{N} 
        % \]
        % with the image of $\hat{\phi}$ isomorphic to $\hat{\sF}$. Since $\sH om(\hat{\sM}^{-M},\hat{\sM}^{N})$ is also a locally free sheaf on $\hat{Y}_{s}$, by Proposition \ref{prop:mod lef cond}, there exists $m$ and $\phi:\sM^{-M}(-mD_{\infty})\rightarrow \sM^{N}$ and $\hat{\phi}$ comes from its completion. Thus take the image of $\phi$ in $\sM^{N}$, whose completion gives rise to $\hat{\sF}$. The proof is completed.
    \end{proof}

    \begin{theorem}\label{thm:eff lef for formal line bundle}
       We have the following exact sequence:
        \[
        0\rightarrow\mathbb{Z}D_{\infty}\to \Pic(Y)\xrightarrow{c} \Pic(\hat Y_s)\to 0.
        \]
        in particular,  if $\dim X\geq 3$,  we have $\pi_{s}^{*}:\Pic(X)\rightarrow \Pic(X_{s})$ is an isomorphism provided $X_{s}$ is smooth.
    \end{theorem}
    \begin{proof}
        By Proposition \ref{prop:inj of formal pic} and Proposition \ref{prop:ker of pic res}, we only need to show that $c$ is surjective. For any formal line bundle $\ssM$ on $\hat Y_s$, by $\tx{Leff}^w(Y,X_s)$ there is an open neighbourhood $V$ of $X_s$ and an invertible sheaf $\sM_V$ on $V$ such that $\hat \sM_V\cong \ssM$. One check that $\sM_V$ can always be extend to a line bundle $\sM$ over $Y$ provided $Y$ is smooth (cf. \cite[Corollary 3.4]{HeinLec} extend $\sM_V$ to a coherent sheaf and take double dual). Thus $c(\sM)=\hat{\sM}\cong\ssM$ and $c$ is surjective. 
        
        If $\dim X\ge 3$, $\Pic(\hat{Y}_{s})\rightarrow \Pic(X_{s})$ is an isomorphism, thus we conclude that $\pi_{s}^{*}:\Pic(X)\rightarrow \Pic(X_{s})$ is an isomorphism.

        % \edit{For injectivity (in positive characteristic), recall there exist a birational contraction \footnote{\edit{This is not true? Because we only assume $\sL$ is ample and globally generated, which need not be very ample, thus we do not have a birational map.}}$g:Y\to Z$ with $\tx{Exc}(g)_{\tx{red}}=D_\infty$. Then one has exact sequence
        % \[\xymatrix{0\ar[r]&\Z D_\infty\ar[r]&\Pic(Y)\ar[r]\ar[rd]^{c}&\Pic(Z)\ar[r]\ar[d]^{c_{Z}}&0\\
        % &&&\Pic(\hat Y_s)&
        % }\]
        % and we then show $c_Z$ is injective. If $\sN_1,\sN_2$ be two invertible sheaves on $Z$ with the same completion along $X_s$, by $\tx{Leff}(Y,X_s)$, there exist a neighbourhood $U$ of $X_s$ with $U\cap \tx{Exc}(g)=\emptyset$ and such that $\sN_1|_U\cong\sN_2|_U$. Here $U$ must has codimension $\geq 2$ in $Z$, because $X_s$ is ample in $Z$. Since $Z$ is normal, we know $\sN_1\cong\sN_2$ by the isomorphism $\Pic(Z)\cong \Pic(U)$. }
    \end{proof}
    
\section{Surface Case} 

In this section, we prove a similar result when $X$ is a smooth surface, assuming that the Picard variety of $X$ is smooth.
\vspace{10pt}

\noindent\bf{Notation}. $\underline{\Pic}_{X/S}$ means the relative Picard functor (the $\fppf$ sheaf of sets) and the scheme it is represented by (if it is representable), if $S=\Spec k$, we may omit $S$. $\Pic(X)$ means the Picard group of $X$, which is isomorphic to $H^1(X,\sO_X^\times)$.     
\vspace{10pt}

Let $X$ be a smooth projective surface over an uncountable algebraically closed field $k$. We assume $X$ can be lifted to  $W_2(k)$, and $\underline{\Pic}^0_X$ is smooth i.e. $\dim_{\Q_\ell}H^1(X_{\et},\Q_\ell)=2\dim H^1(X,\sO_X)$, $\ell\neq \chr(k)$. As before, let $\sL$ be a very ample line bundle on $X$ and we put $Y=\bf P(\sL^{\vee}\oplus \sO)$ with $\pi:Y\to X$ the natural projection. Since $\sL$ is very ample, we have $Z=g(Y)$ consisting a unique singularity $o=\phi(D_{\infty})$ where $g$ is induced by the complete linear system $|\pi^*\sL^{r}\otimes \sO_{Y/X}(r)|=|\sW|$. 

Let $X_s\subset Y$ be a smooth spectral surface defined by $s\in |\pi^*\sL^{r}\otimes \sO_{Y/X}(r)|$, which does not intersect with $D_\infty$. One can always view $X_s$ as a very ample divisor in $Z$ defined by a global section of $\sO_Z(1)$. The main goal of this subsection is to prove the following theorem:
\begin{theorem}[$\chr >0, \dim X=2$ case] For very general $s\in |\pi^*\sL^{r}\otimes \sO_{Y/X}(r)|$ the map $\pi^{*}:\underline{\Pic}_X\cong \underline{\Pic}_{X_s}$ is an isomorphism.
\end{theorem}
Since $Z$ is normal and has a unique singularity, we have:
\blemma $\Pic(X)\cong \Pic({\bf V_X(\sL^\vee)})\cong \Pic(U)=\Pic(Z)$.
\elemma
% \begin{proof}
%     Because $Z$ is normal and $Z\cong \bf V_X(\sL^\vee)\cup \{o\}$.

% \end{proof}
Since $X_{s}$ can be treated as a closed subvariety of $Z$, we put $\hat{Z}_{s}$ as the formal completion of $Z$ along $X_{s}$. In fact, we have $\hat{Z}_{s}\cong\hat{Y}_{s}$.
\blemma For any $s$ parametrizing smooth spectral surface $X_s$, we have $\Pic(Z)\inj \Pic(\hat Z)\inj \Pic(X_s)$.
\elemma
\begin{proof}
    In Proposition \ref{prop:ker of pic res}, we have proved the exact sequence: $0\rightarrow\mathbb{Z}D_{\infty}\rightarrow \Pic(Y)\rightarrow \Pic(\hat{Y}_{s})$. Since $g:Y\rightarrow Z$ is a contraction of $D_{\infty}$ and $Z$ is normal, $0\rightarrow\mathbb{Z}D_{\infty}\rightarrow \Pic(Y)\rightarrow \Pic(Z)\rightarrow 0$. And $\Pic(\hat{Y_{s}})\cong \Pic(\hat{Z}_{s})$, then by Proposition \ref{prop:inj of formal pic}, $\Pic(\hat{Z})\hookrightarrow \Pic(X_{s})$.
\end{proof}

% This is because $X$ can be lift to $W_2(k)$ thus Kodaira -vanishing theorem holds.

\begin{lemma}\label{lem:smoothness of pic} For any $s$ parametrizing smooth spectral surface $X_s$, then $\underline{\Pic}(X_{s})$ is also smooth, and we have  $\underline{\Pic}_X^0=\underline{\Pic}^0_Z\cong \underline{\Pic}^0_{X_s}$.
\end{lemma}
\begin{proof}We have the composite
\[\Pic(X)\xrightarrow[\cong]{\pi_{\bf V(\sL^\vee)/X}^*}\Pic(\bf V(\sL^{\vee}))\cong \Pic(Z)\xrightarrow[\inj]{i_{X_s}^*}\Pic(X_s)\]
equals to the flat pullback $\pi_{s}^*:\Pic(X)\to \Pic(X_s)$, so $\pi^*_{s}$ is injective, in particular injective on $\ell$-torsion points. i.e. $\pi^*_{s}:H^1(X_{\et},\mu_\ell)\inj H^1(X_{s\ \et},\mu_\ell)$. Thus one has the comparison of $\Z/\ell\Z$-Betti numbers $b_1(X)\leq b_1(X_s)$.
  
Since $\pi_{s*}\sO_{X_s}\cong\oplus_{i=0}^{r-1} \sL^{\otimes -i}$, then by the Kodaira vanishing theorem $$\pi^*_{s}:H^1(X,\sO_X)\to H^1(X_s,\sO_{X_s})$$ is an isomorphism. By our assumption of smoothness of $\underline{\Pic}^0_{X}$, we have $b_1(X)=2h^1(X,\sO_X)$ (in fact, this also follows from $W_{2}$ lifting). Combined with the previous results, we have $$b_1(X)\leq b_1(X_s)\leq 2h^1(X_s,\sO_{X_s})= 2h^1(X,\sO_X)=b_1(X).$$
Hence we get the smoothness of $\underline{\Pic}^0_{X_s}$, and thus the smoothness of $\underline{\Pic}_{X_s}$ by \cite[Corollary 9.5.13]{FGA}, and isomorphisms $\underline{\Pic}_X^0=\underline{\Pic}^0_Z\cong \underline{\Pic}^0_{X_s}$.
\end{proof}

%Let $\sY\subset \sX:=Y\times \bf P(H^{0}(Y,\pi^{*}\sL^{n}\otimes\sO(n))^{\vee})$ be the universal family of divisors parametrized by $\bf P(H^{0}(Y,\pi^{*}\sL^{n}\otimes\sO(n))^{\vee})$.
%\[
%\xymatrix{&\sY\ar[ld]^{p}\ar[rd]_{q}\\
%Y&&\bf P(H^{0}(Y,\pi^{*}\sL^{n}\otimes\sO(n))^{\vee})}
%\]

Let $\sY$ contained in $Y\times \bf P(H^{0}(Y,\pi^{*}\sL^{r}\otimes\sO(r))^\vee):=\bf P_Y$ be the universal family of divisors parametrized by $\bf P(H^{0}(Y,\pi^{*}\sL^{r}\otimes\sO(r))^\vee)$. Let us denote by the projections restricted to $\sY$ by $p=\bm p_Y|_{\sY}:\sY\to Y$ and $q:=\bm p_{\bf P}|_{\sY}:\sY\to \bf P(H^{0}(Y,\pi^{*}\sL^{r}\otimes\sO(r))^\vee)$. By the construction, $\sY\subset \bf P_Y$ is relatively very ample to $p_Y$. Thus $R^i{p_Y}_*\sO_{\bf P_Y}(-\sY)=0$ for all $i\geq 1$.

For a closed point $s\in\bm A\subset \bf P(H^{0}(Y,\pi^{*}\sL^{r}\otimes\sO_{Y/X}(r))^\vee)$, we write $\sY_{s}:=q^{-1}(s)$  which is isomorphic to $X_{s}$ under the projection $p$. Similar as before, we put $\sY_{s,\ell}$ the $\ell$-th thickening of $\sY_{s}$, and $\hat{\sY}_{s}$ the formal completion along $\sY_{s}$. We denote the maximal ideal of $s$ by $\go m_{s}$, thus the defining ideal of $\sY_{s}$ is $q^{*}\go m_{s}=\go m_{s}\sO_{\sY}$. We denote the defining ideal of $X_{s}$ in $Y$ by $\sI_{s}\cong \sW^{-1}$. Then we have $p^{*}\sI_{s}\subset q^{*}\go m_{s}$. In particular, we have a map of infinitesimal thickenings $X_{s,\ell}\rightarrow \sY_{s,\ell}$ induced by $p$. As a result, we have the following commutative diagram:
\[
\xymatrix{&\Pic(\hat{Y}_{s})\ar[dd]\ar[r]&\Pic(X_{s,\ell})\ar[dd]\ar[rd]\\
\Pic(Y)\ar[rd]\ar[ru]&&&\Pic(X_{s})\\
&\Pic(\hat{\sY}_{s})\ar[r]&\Pic(\sY_{s,\ell})\ar[ru]
}
\]

Following \cite{RS09}, we introduce the infinitesimal Noether-Lefschetz condition.
\begin{definition}\cite{RS09}
    We say the pair $(Y,X_{s})$ satisfies the $\ell$-th infinitesimal Noether-Lefschetz condition, denoted by $\text{INL}_{\ell}$ if the following map is an isomorphism:
    \[
    \text{Image}(\Pic(X_{s,\ell})\rightarrow\Pic(X_{s}))\rightarrow\text{Image}(\Pic(\sY_{s,\ell})\rightarrow\Pic(X_{s}))
    \]
    Similarly, we say the pair $(Y,X_{s})$ satisfies the formal Noether-Lefschetz condition, denoted by $\text{FNL}$ if the following map is an isomorphism:
    \[
    \text{Image}(\Pic(\hat{Y}_{s})\rightarrow\Pic(X_{s}))\rightarrow\text{Image}(\Pic(\hat{\sY}_{s})\rightarrow\Pic(X_{s}))
    \]
\end{definition}

We restrict the universal family to get a smooth family. Without causing ambiguity, we still denote it by $q:\sY_{U}\rightarrow U$ where $q$ is smooth with connected fibers. Consider the relative Picard functor $\underline{\Pic}_{q}$, $q$ is flat, projective with integral geometric fibers, so by \cite[Theorem 9.4.8]{FGA}, $\underline{\Pic}_{q}$ is representable by a scheme which is separated and locally of finite type over $U$, and represents $\underline{\Pic}_{q}^\et$, the \'etale sheaf associated with $\underline{\Pic}_{q}$·

By the base change property of relative Picard scheme,
 for any closed point $\xi\in U$, $\underline{\Pic}_{q}|_{\xi}=\underline{\Pic}_{X_\xi/\xi}$ is smooth. It means that each fiber of $\underline{\Pic}_{q}\rightarrow U$ is smooth.
% \begin{proof}

% As in the proof of Lemma \ref{lem:smoothness of pic}, $R^1q_{U*}\sO_{\sY_{U}}$ is locally free (and has base change property), so $\dim \underline{\Pic}_{X_\xi/\xi}=\dim H^1(X_\xi,\sO_{X_\xi})$ and by \cite[Corollary 9.5.13]{FGA} we get the smoothness. 
% \end{proof}

Let $\tx{Hilb}$ be the set of Hilbert polynomials of line bundles on $\sY_{U}$. $\Phi\subset \tx{Hilb}$ be a finite subset. Denote by $\underline{\Pic}_q^\Phi\subset \underline{\Pic}_q$ be the components with Hilbert polynomials in $\Phi$, then $\underline{\Pic}_q^\Phi$ is of finite type over $U$. Denote by $U^\Phi\subset U$ the open subset of $U$ on which $\underline{\Pic}_q^\Phi$ is flat (hence smooth). This subset is non-empty because it contains the generic point of $U$. The intersection of $U^\Phi$ for $\Phi$ covering $\tx{Hilb}$ is a very general subset $\bm V$ of $U$.  

\blemma
If for any $s\in \bm V\subset U$ and any $m_1>m_2\in \Z_{>0}$ , we have $\pi^{*}:\Pic(\hat{\sY}_{s,m_1})\rightarrow\Pic(\hat{\sY}_{s,m_2})$ is surjective.
\elemma
\begin{proof}
	For $s\in \bm V\subset U$, we have $\underline{\Pic}_{q}$ is smooth over $s$. By the smoothness, we have the following surjective morphisms, \begin{footnotesize}
		$$\Hom(\Spec\hat{\sO}_{U,s}, \underline{\Pic}_{q})\surj \Hom(\Spec(\sO_{U,s}/\go m_{s}^{m_1}),\underline{\Pic}_{q})\surj \Hom(\Spec(\sO_{U,s}/\go m_{s}^{m_2}),\underline{\Pic}_{q})$$.
	\end{footnotesize} Consider the following exact sequence, coming from low-degree terms in the Leray spectral sequence \cite[Chapter III, Theorem 1.18]{Mil80} for $\GG_m$ relative to $q_{S}$ with $S= \Spec\hat{\sO}_{U,s}$, (see also \cite[Chapter 9, (9.2.11.5)]{FGA})
$$0\to \Pic(S)\to \Pic(\sY_S)\to \underline{\Pic}^\et_{q}(S)\to H^2(S_\et,{q_S}_*\GG_m).$$
%$0\to H^1 (T, {f_T}_*\GG_{m,\tilde Z_T})\to H^1 (\tilde Z_T, \GG_{m,\tilde{Z}_T})\to H^0 (T, R^1{f_T}_*\GG_m)\to H^2 (T, {f_T}_*\GG_{m,\tilde Z_T})$
%$\GG_m\subset \GG_a$
%$f_*\GG_m\inj f_*\GG_a$, $f_*\sO_{\tilde Z_{U_s}}=\sO_{U_s}$. 
We have ${q_S}_*\GG_m=\GG_{m,S}$. Thus the previous sequence becomes
$$0\to \Pic(S)\to \Pic(\sY_S)\to \underline{\Pic}^\et_{q}(S)\to H^2(S_{\et},\GG_m).$$
%In particular, for $S$, we have
%$0\to 0\to \Pic(\tilde Z_{\hat{\sO}_{U_s,s}})\to \underline{\Pic}^\et_{f}(\Spec(\hat{\sO}_{U_s,s}))\to H^2(\Spec(\hat{\sO}_{U_s,s}),\GG_m).$
By \cite[Chapter III, Theorem 3.9]{Mil80}, $S$ is strictly local, then $H^2(S_\et,\GG_m)=0$, we have $\underline{\Pic}_q^{\et}(S)\cong \Pic(\sY_S)$ for $S= \Spec\hat{\sO}_{U,s}$. Since ${\sO}_{U,s}/\go m_{s}^\ell$ for $\ell\in\Z_{>0}$ is an Artin local ring with algebraically closed residue field, we have $\underline{\Pic}_q^{\et}(S)\cong \Pic(\sY_S)$ for $S=\Spec{\sO}_{U,s}/\go m_{s}^\ell$ for $\ell\in\Z_{>0}$. And the previous surjections show that $\pi^{*}:\Pic(\sY_{s,m_1})\rightarrow\Pic(\sY_{s,m_2})$ is surjective for any $s\in \bm V$ and $m_1>m_2\in \Z_{>0}$.
\end{proof}

It is easy to deduce the following corollary,
\begin{corollary}\label{cor:FNL holds then}
	For any closed point $s\in \bf V$, if the pair $(Y,X_{s})$ satisfies the $\text{FNL}$, then we have $\pi^{*}:\Pic(X)\rightarrow\Pic(X_{s})$ is an isomorphism.
\end{corollary}

% \begin{proof}
%     By Proposition \ref{prop:ker of pic res}, we only need to show that $\pi^{*}:\Pic(X)\rightarrow\Pic(X_{s})$ is surjective. For a very general $s\in U$, we have $\underline{\Pic}_{q}$ is smooth over $s$. By the smoothness,  $\Hom(\Spec\hat{\sO}_{U,s}, \underline{\Pic}_{q})\rightarrow \Hom(\Spec(k_{s}),\underline{\Pic}_{q})$ is surjective. Notice that any line bundle on $\Pic(X_{s})$ corresponds an element in $\Hom(\Spec(k_{s}),\underline{\Pic}_{q})$, thus it can be lifts to an element in $\Hom(\Spec\hat{\sO}_{U,s}, \underline{\Pic}_{q})$. Since $\underline{\Pic}_{q}$ represents the \'etale sheaf restriction, and any connected \'etale cover of $\Spec\hat{\sO}_{U,s}$ is itself, thus an element in $\Hom(\Spec\hat{\sO}_{U,s}, \underline{\Pic}_{q})$ corresponds to a line bundle on $\hat{\sY}_{s}$. Then by the $\text{FNL}$, a line bundle on $X_{s}$ can be extended to a formal line bundle over $\hat{Y}_{s}$. In other words, we have a surjection $\Pic(\hat{Y}_{s})\twoheadrightarrow \Pic(X_{s})$. Then by Theorem \ref{thm:eff lef for formal line bundle}, we show that $\pi^{*}:\Pic(X_{s})\twoheadrightarrow \Pic(X_{s})$.
% \end{proof}

We here present several cohomological results from \cite{RS09} which we shall need later. 
\begin{lemma}\cite[Lemma 1 and  2.1]{RS09}\label{lem:direct image comp}
\begin{enumerate}
    \item $Rp_{*}(q^{*}\go m_{s})\cong \sI_{s}$;
    \item $p_{*}(q^*\go m_{s}^{\ell})=\sI_{s}^{\ell}$ for $\ell\ge 1$;
    \item $0\rightarrow\sO_{X_{s,\ell}}\rightarrow p_{*}\sO_{\sY_{s,\ell}}\rightarrow R^{1}p_{*}(q^{*}\go m_{s}^{\ell})\rightarrow 0$ .
    \item $R^jp_*q^*\go m^\ell_s=0$ for $j\geq 2$ and $\ell\in\Z_+$. 
\end{enumerate}
\end{lemma}

The following proposition is crucial for us to prove the $\text{FNL}$ for $s\in \bf V\subset U$.
\begin{proposition}\label{prop:inj for every step}
     \[
     H^{2}(X_{s}, \sI_{s}^{\ell}/\sI_{s}^{\ell+1})\rightarrow H^{2}(\sY_{s},q^{*}\go m_{s}^{\ell}/q^{*}\go m_{s}^{\ell+1})
     \]
    is injective for all $\ell>0$.
\end{proposition}
\begin{proof}
     Consider the following commutative diagram:
    \[
    \xymatrix{
    0\ar[r]&\sI_{s}^{\ell+1}\ar[r]\ar[d]&\sI_{s}^{\ell}\ar[r]\ar[d]&\sI_{s}^{\ell}/\sI_{s}^{\ell+1}\ar[r]\ar[d]&0\\
    0\ar[r]&q^{*}\go m_{s}^{\ell+1}
    \ar[r]&q^{*}\go m_{s}^{\ell}\ar[r]&q^{*}\go m^{\ell}/q^{*}\go m_{s}^{\ell+1}\ar[r]&0}
    \]
    Recall that $\sI_{s}\cong\sW^{-1}=\pi^{*}\sL^{-r}\otimes\sO_{Y/X}(-r)$. By Lemma \ref{lem: a vanishing theorem},  $H^{i}(Y, \sI_{s}^{\ell})=0$ for all $\ell>0, i=1,2$.
    Thus we have:
     \begin{footnotesize}
        \[
    \xymatrix{
    0\ar[r]&H^{2}(X_{s},\sI_{s}^{\ell}/\sI_{s}^{\ell+1})\ar[r]\ar[d]&H^{3}(Y,\sI_{s}^{\ell+1})\ar[r]\ar[d]&H^{3}(Y,\sI_{s}^{\ell})\ar[d]\\
    \ar[r]&H^{2}(\sY_{s},q^{*}\go m_{s}^{\ell}/q^{*}\go m_{s}^{\ell+1})\ar[r]&H^{3}(\sY,q^{*}\go m_{s}^{\ell+1})\ar[r]&H^{3}(\sY,q^{*}\go m_{s}^{\ell})
    }
    \]
    \end{footnotesize}
    Then we only need to prove that $H^{3}(Y,\sI_{s}^{\ell})\rightarrow H^{3}(\sY,q^{*}\go m_{s}^{\ell})$ is injective for all $\ell>0$. Since $p_{*}(q^{*}m_{s}^{\ell})=\sI^{\ell}_s$, and considering the Leray spectral sequence for $p:\sY\rightarrow Y$, $H^{3}(Y,\sI_{s}^{\ell})\rightarrow H^{3}(\sY,q^{*}\go m_{s}^{\ell})$ is the map $E_{2}^{3,0}\rightarrow H^{3}(\sY,q^{*}\go m^{\ell})$, thus to show the injectivity, we only need to show the following differential vanishes for all $\ell>0$:
    \[
    H^{1}(Y,R^{1}p_{*}q^{*}\go m_{s}^{\ell})=E_{2}^{1,1}\rightarrow E_{2}^{3,0}=H^{3}(Y,\sI_{s}^{\ell})
    \]
    By the Lemma \ref{lem: a necessary surjectivity} below (which is due to Ravindra and Srinivas \cite{RS09}), it amounts to saying that the map:
\[
    H^{0}(Y,\sW^{\otimes (\ell -1)})\otimes H^{0}(Y,\omega_{Y}\otimes\sW)\rightarrow H^{0}(Y,\omega_{Y}\otimes\sW^{\otimes \ell})
    \]
    is surjective for all $\ell>0$. Since $ H^{0}(Y,\sW^{\otimes (\ell -1)})=H^{0}(Z,\sO_{Z}(\ell-1))$, this follows from the Corollary \ref{cor:surjection we need} as a result of the $0$-regularity of the sheaf $g_{*}(\omega_{Y}\otimes \sW)$ on $Z=g(Y)$ with respect to the very ample line bundle $\sH$ on $Z$. 
    
%      By the zero-regularity of $g_*(\omega_Y\otimes\sW)$ on $Z$ with respect to $\sH$, we have the surjective maps
%   \begin{small}
%       $$H^0(Z,g_*(\omega_Y\otimes\sW))\otimes H^0(Z,\sH^{\otimes(r\ell-r)})\surj H^0(Z,g_*(\omega_Y\otimes\sW)\otimes\sH^{\otimes (r\ell)})$$
%       \end{small} for all $\ell>0$, which is the surjection in \ref{lem: a necessary surjectivity}.
\end{proof}
The following crucial lemma is proved in \cite[Section 2.1,2.2]{RS09}. Here for reader's convenience, we give a sketch of their proof.
\begin{lemma}\label{lem: a necessary surjectivity}
    We denote $\pi^*\sL^{\otimes r}\otimes\sO_{Y/X}(r)$ by $\sW$. The surjectivity of 
 \begin{small}
         \[
    H^{0}(Y,\sW^{\otimes(\ell-1)})\otimes H^{0}(Y,\omega_{Y}\otimes\sW)\rightarrow H^{0}(Y,\omega_{Y}\otimes\sW^{\otimes \ell})
    \] 
 \end{small}
    implies the vanishing of $H^{1}(Y,R^{1}p_{*}q^{*}\go m_{s}^{\ell})\rightarrow H^{3}(Y,\sI_{s}^{\ell})$ for $\ell\ge 1$ .
\end{lemma}
\begin{proof}
    (\textit{Completely follows from} \cite[Section 2.1,2.2]{RS09}). Since $\bf P_Y:=Y\times \bf P(H^{0}(Y,\pi^{*}\sL^{r}\otimes\sO(r))^{\vee})$ is a trivial projective bundle over $Y$, we have ${\bm p_Y}_{*}\go m_{s}^{\ell}\sO_{\bf P_Y}=0$ for any $\ell>0$. Applying $R{\bm p_Y}_{*}$ to the exact sequence:
    \[
    0\rightarrow \go m^\ell_s \sO_{\bf P_Y}(-\sY) \rightarrow \go m_{s}^{\ell}\sO_{\bf P_Y}\rightarrow \go m_{s}^{\ell}\sO_{\sY}(=q^*\go m^\ell_s)\rightarrow 0
    \]
    since $p_{*}q^*\go m_{s}^{\ell}=\sI_{s}^{\ell}$ and $R^jp_*q^*\go m_s^\ell=0$ for all $j\geq 2$, we have 
    \[
    0\rightarrow \sI_{s}^{n}\rightarrow R^{1}{\bm p_Y}_{*}(\go m_s^\ell\sO_{\bf P_Y}(-\sY))\rightarrow R^{1}{\bm p_Y}_{*}(\go m_{s}^{\ell}\sO_{\bf P_Y})\rightarrow R^{1}p_{*}(\go m_{s}^{\ell}\sO_{\sY})\rightarrow 0
    \]
    If we split the above four term exact sequence into two short exact sequences:
    \begin{align*}
        0\rightarrow \sI_{s}^{\ell}\rightarrow R^{1}{\bm p_Y}_{*}(\go m_s^\ell\sO_{\bf P_Y}(-\sY))\rightarrow \sF_{\ell}\rightarrow 0\\
        0\rightarrow \sF_{\ell}\rightarrow R^{1}{\bm p_Y}_{*}(\go m_s^\ell\sO_{\bf P_Y})\rightarrow R^{1}p_{*}(\go m_{s}^{\ell}\sO_{\sY})\rightarrow 0
    \end{align*}
    thus the differential:
    \[
    H^{1}(Y,R^{1}p_{*}q^{*}\go m_{s}^{\ell})=E_{2}^{1,1}\rightarrow E_{2}^{3,0}=H^{3}(Y,\sI_{s}^{\ell})
    \]
    factors through $H^{1}(Y,R^{1}p_{*}q^{*}\go m_{s}^{\ell})\rightarrow H^{2}(Y,\sF_{\ell})\rightarrow H^{3}(Y,\sI_{s}^{\ell})$. Then it vanishes for all $\ell>0$, if
    \begin{equation}\label{eq: order 3 inj}
        H^{3}(Y,\sI_{s}^{\ell})\rightarrow H^{3}(Y,R^{1}{\bm p_Y}_{*}(\go m_s^\ell\sO_{\bf P_Y}(-\sY)))
    \end{equation} is injective. 
    
    We denote $\Spec(\sO_{\bf P,s}/\go m_S^\ell)$ by $s_\ell$. Now applying $R{\bm p_Y}_{*}$ to:
    \[
    0\rightarrow \go m_s^\ell\sO_{\bf P_Y}(-\sY)\rightarrow \sO_{\bf P_Y}(-\sY)\rightarrow \sO_{s_\ell\times Y}(-\sY_{s,\ell})\rightarrow 0,
    \]
    we have $R^{1}{\bm p_Y}_{*}(\go m_s^\ell\sO_{\bf P_Y}(-\sY))\cong p_{*}\sO_{s_\ell\times Y}(-\sY_{s,\ell})$. For notation ease, recall that we put $W= H^{0}(Y,\pi^{*}\sL^{r}\otimes\sO_{Y/X}(r))=H^0(Y,\sW)$. By Kunneth formula $$p_{*}\sO_{s_\ell\times Y}(-\sY_{s,\ell})\cong_{\sO_Y\text{-mod}} H^{0}(\bf P(W), \sO_{\bf P(W)}(-1)/\go m_{s}^{\ell})\otimes_{k}\sI_{s}.$$ Thus the map \eqref{eq: order 3 inj} is:
    \begin{equation}
        H^{3}(Y,\sI_{s}^{\ell})\rightarrow H^{3}(Y,\sI_{s})\otimes H^{0}(\bf P(W), \sO_{\bf P(W)}(-1)/\go m_{s}^{\ell})
    \end{equation}
    which is injective. By Serre duality, it amounts to saying that:
       \begin{equation}\label{eq: order 0 surj}
    H^{0}(Y,\omega_{Y}\otimes\sW)\otimes H^{0}(\bf P(W), \sO_{\bf P(W)}(-1)/\go m_{s}^{\ell})^{\vee}\surj H^{0}(Y,\omega_{Y}\otimes\sW^{\otimes \ell})
    \end{equation}
    is surjective.
    
    We can check that \begin{footnotesize}
        $$ H^{0}(\bf P(W), \sO_{\bf P(W)}(-1)/m_{s}^{\ell})^{\vee}\cong \Sym^{\ell-1} (W)=H^0(\bf P(W),\sO_{\bf P(W)}(\ell-1))=H^0(Y,\sW^{\otimes(\ell-1)})$$
    \end{footnotesize} and we have the surjective evaluation map:
    \[
    \Sym^{\ell-1} (W)\otimes \sO_{\bf P(W)}\surj \sO_{\bf P(W)}(\ell-1)
    \] and its pulling back $\Sym^{\ell-1}(W)\otimes\sO_Y\xrightarrow[\tx{surj.}]{ev_{\ell-1}^{\sW}} \sW^{\otimes(\ell -1)}$ on $Y$. Then the surjection \eqref{eq: order 0 surj} becomes:
    $$H^0(Y,\omega_Y\otimes\sW)\otimes H^0(Y,\sW^{\otimes(\ell-1)})\surj H^0(Y,\omega_Y\otimes\sW^{\otimes \ell}).$$
    
    %%By the zero-regularity of $g_*(\omega_Y\otimes\sW)$ on $\bf P(W)$ with respect to $\sH$, we have the surjective maps
 % \begin{small}
%      $$H^0(\bf P(W),g_*(\omega_Y\otimes\sW))\otimes H^0(\bf P(W),\sH^{\otimes(r\ell-r)})\surj H^0(\bf P(W),g_*(\omega_Y\otimes\sW)\otimes\sH^{\otimes (r\ell-r)})$$
%  \end{small} for all $\ell$ and hence
%    \begin{footnotesize}
 %      \[\Phi:\Sym^{\ell-1}(W)\otimes H^0(Y,\omega_Y\otimes\sW)\otimes\sO_Y\xrightarrow[\text{surj.}]{ev_{\ell-1}^\sW\otimes\id}H^0(Y,\omega_Y\otimes\sW)\otimes\sW^{\otimes(\ell-1)}\xrightarrow[\text{surj.}]{ev_1^{\omega_Y\otimes\sW}\otimes\id_{\sW^{\otimes(\ell-1)}}} \sW^{\otimes\ell} \]
 %   \end{footnotesize}
%Taking global section of $\Phi$,  we get the surjection .
%    Thus the map \eqref{eq: order 0 surj} is obtained as follows:
%    \begin{footnotesize}
%    \[
%    \xymatrix{
%    \sW^{\otimes(\ell-1)}\otimes H^{0}(Y,\omega_{Y}\otimes\sW)\ar[r]&\\
%    H^{0}(Y,\pi^{*}\sL^{\ell-1}\otimes\sO_{Y/X}(\ell-1))\otimes H^{0}(Y,\omega_{Y}\otimes \pi^{*}\sL\otimes\sO_{Y/X}(1)))\ar[ru]\ar[u]^{\phi^{*}}
%    }
%    \]
%    \end{footnotesize}
%    Recall that the linear system $|\pi^{*}\sL^{n}\otimes\sO_{Y/X}(n)|$ defines a map $\phi:Y\rightarrow \bf P(W)$. 
\end{proof}
The following proposition is an adapted version of that in \cite[Proposition 1 ]{RS09}, under the assumption that $\chr(k)=p\ge 3$, $X$ admits a $W_{2}(k)$-lifting and $\underline{\Pic}_{X}^{0}$ is smooth.
\begin{proposition}\label{prop:FNL}
When $r>3$, then for any closed point $s\in \bf V$ (i.e. over $s$ the relative Picard variety is smooth), we have the pair $(Y,X_{s})$ satisfies $\text{INL}_{m}$ for all $m>0$ and then it satisfies $\text{FNL}$.
\end{proposition}
\begin{proof}
    For any $m>0$, we have:
    \[    \xymatrix{
    0\ar[r]& \sI_{s}^{m}/\sI_{s}^{m+1}\ar[r]\ar[d]|{\beta_{m+1}}&\sO^{\times}_{X_{s,m+1}}\ar[r]\ar[d]|{p^{\flat}_{m+1}}&\sO^{\times}_{X_{s,m}}\ar[r]\ar[d]|{p^{\flat}_{m}}&0\\
     0\ar[r]& q^*\go m_{s}^{m}/q^*\go m_{s}^{m+1}\ar[r]&\sO^{\times}_{\sY_{s,m+1}}\ar[r]&\sO^{\times}_{\sY_{s,m}}\ar[r]&0
    }.\]
    Since $H^{1}(X_{s},\sI_{s}^{m}/\sI_{s}^{m+1})=0$, we have the following exact sequence:
    \begin{footnotesize}\begin{equation}
        \xymatrix@C=3ex{
        0\ar[r]&\Pic(X_{s,m+1})\ar[r]\ar[d]&\Pic(X_{s,m})\ar[r]\ar[d]& H^{2}(X_{s}, \sI_{s}^{m}/\sI_{s}^{m+1})\ar[d]\\
        0\ar[r]&\Pic(\sY_{s,m+1})/H^{1}(\sY_{s},q^{*}\go m_{s}^{m}/q^{*}\go m_{s}^{m+1})\ar[r]&\Pic(\sY_{s,m})\ar[r]& H^{2}(\sY_{s},q^{*}\go m_{s}^{m}/q^{*}\go m_{s}^{m+1})
        }
    \end{equation}\end{footnotesize}
    Notice that $s\in \bf V$, we have the surjection $\Pic(\sY_{s,m+1})\surj \Pic(\sY_{s,m})$, which implies that the map $\Pic(\sY_{s,m})\rightarrow H^{2}(\sY_{s},q^{*}\go m_{s}^{m}/q^{*}\go m_{s}^{m+1})$ is a zero map. By the injectivity of $H^{2}(X_{s}, \sI_{s}^{m}/\sI_{s}^{m+1})\inj H^{2}(\sY_{s},q^{*}\go m_{s}^{m}/q^{*}\go m_{s}^{m+1})$ for all $m>0$, we know that $\Pic(X_{s,m})\to H^{2}(X_{s}, \sI_{s}^{m}/\sI_{s}^{m+1})$ is a zero map. Thus $\Pic(X_{s,m+1})\to\Pic(X_{s,m})$ is surjective. As a result we have $\text{Image}(\Pic(X_{s,m})\rightarrow\Pic(X_{s}))=\Pic(X_{s})$ and the $\text{FNL}$ holds.
    %  Since $H^{1}(X,\sO_{X})=0$, we have $H^{1}(X_{s},\sO_{X_{s}})=0$, thus $H^{1}(\sY_{s},q^{*}\go m_{s}^{m}/q^{*}\go m_{s}^{m+1})=0$. Thus the above diagram takes the following form:
    %  \begin{footnotesize}\begin{equation}
    %     \xymatrix@C=3ex{
    %     0\ar[r]&\Pic(X_{s,m+1})\ar[r]\ar[d]&\Pic(X_{s,m})\ar[r]\ar[d]& H^{2}(X_{s}, \sI_{s}^{m}/\sI_{s}^{m+1})\ar[d]\\
    %     0\ar[r]&\Pic(\sY_{s,m+1})\ar[r]&\Pic(\sY_{s,m})\ar[r]& H^{2}(\sY_{s},q^{*}\go m_{s}^{m}/q^{*}\go m_{s}^{m+1})
    %     }
    % \end{equation}\end{footnotesize}
    %  Notice that when $m=o$, $X_{s,m}\cong \sY_{s,m}$. Since $$H^{2}(X_{s}, \sI_{s}^{\ell}/\sI_{s}^{\ell+1})\rightarrow H^{2}(\sY_{s},q^{*}\go m_{s}^{\ell}/q^{*}\go m_{s}^{\ell+1})$$ is injective for all $\ell<m$, then $$\text{Image}(\Pic(X_{s,\ell+1})\rightarrow\Pic(X_{s,\ell}))\rightarrow\text{Image}(\Pic(\sY_{s,\ell+1})\rightarrow\Pic(\sY_{s,\ell}))$$ is surjective for all $\ell<m$ by induction, and as a result we have $$\text{Image}(\Pic(X_{s,m})\rightarrow\Pic(X_{s}))\rightarrow\text{Image}(\Pic(\hat{\sY}_{s,m})\rightarrow\Pic(X_{s}))$$ is an isomorphism, i.e. we have $\text{INL}_{m}$. 
    \end{proof}
    
    \begin{remark}
    Our proof almost follows from \cite{RS09} except that in positive characteristic, we do not have the exponential sequence for $\ell\ge p$:
    \[
    0\rightarrow \sI_{s}/\sI_{s}^{\ell}\xrightarrow{\exp}\sO^{\times}_{X_{s,\ell}}\rightarrow\sO^{\times}_{X_{s}}\rightarrow 0.
    \]
    To compensate this, we build up the $\text{FNL}$ step by step and thus we have to use the stronger property that $g_{*}(\omega_{Y}\otimes\pi^{*}\sL^{\otimes r}\otimes \sO_{Y/X}(r))$ is $0$-regular. In characteristic 0, by \cite[Theorem 2]{RS09}, to satisfy the $\text{FNL}$, we only need $g_{*}(\omega_{Y}\otimes\pi^{*}\sL^{\otimes r}\otimes \sO_{Y/X}(r))$ is globally generated.
    \end{remark}
    To conclude, we have
    \begin{theorem}
        When $r\ge 4$, $\underline{\Pic}^0_{X}$ is smooth, for $s\in \bm V$, we have the isomorphism $\pi^{*}:\Pic(X)\rightarrow\Pic(X_{s})$.
    \end{theorem}
    Notice that, if $X=\mathbb{P}^{2}$, $\sL=\sO(1)$, and $r=3$, then $X_{s}$ is a cubic surface and thus its Picard number is different from the Picard number of $X$. However, if the canonical line bundle of $X$ is sufficient ample, we expect that the theorem holds for smaller $r$.
    \begin{remark}\label{rmk:comparison with Ji's result}
    We can also use the recent deep result of Lena Ji about Noether-Lefschetz theorem on normal threefold in positive characteristics to obtain the theorem for $r=4$ or $r\ge 6$ without even assuming $W_{2}$ lifting and smoothness of $\underline{\Pic}^{0}_{X}$. We here introduce the adapted proof of that in \cite{RS09} to provide a self-contained proof. Besides, our results provides various normal varieties, i.e., $Z=g(Y)$, where $r\ge 4$ is sufficient. 
    \end{remark}
    \subsection{More on ``Formal Noether-Lefschetz" Conditions}
    In the Proposition \ref{prop:FNL}, we show that for  $s\in \bm V$, the $\text{FNL}$ holds, and thus we can compute the Picard group of the corresponding spectral surface $X_{s}$. In the following lemma, we want to point out that the $\text{FNL}$ may holds over $U$ (possibly much larger than $\bm V$) parametrizing smooth spectral surfaces under an extra condition $H^{1}(X,\sO_{X})=0$.
    \begin{lemma}\label{lem:more on FNL}
        When $r>3$, and $H^{1}(X,\sO_{X})=0$, then for any closed point $s\in U$ (i.e. over which the spectral surface is smooth), we have the pair $(Y,X_{s})$ satisfies $\text{INL}_{m}$ for all $m>0$ and then it satisfies $\text{FNL}$.
    \end{lemma}
    \begin{proof}
       We still consider the following two exact sequences.  For any $m>0$, we have:
    \[
    \xymatrix{
    0\ar[r]& \sI_{s}^{m}/\sI_{s}^{m+1}\ar[r]\ar[d]&\sO^{\times}_{X_{s,m+1}}\ar[r]\ar[d]&\sO^{\times}_{X_{s,m}}\ar[r]\ar[d]&0\\
     0\ar[r]& q^*\go m_{s}^{m}/q^*\go m_{s}^{m+1}\ar[r]&\sO^{\times}_{\sY_{s,m+1}}\ar[r]&\sO^{\times}_{\sY_{s,m}}\ar[r]&0
    }.
    \]
    Since $H^{1}(X_{s},\sI_{s}^{m}/\sI_{s}^{m+1})=0$, we have the following exact sequence:
    \begin{footnotesize}\begin{equation}
        \xymatrix@C=2ex{
        0\ar[r]&\Pic(X_{s,m+1})\ar[r]\ar[d]&\Pic(X_{s,m})\ar[r]\ar[d]& H^{2}(X_{s}, \sI_{s}^{m}/\sI_{s}^{m+1})\ar[d]\\
        0\ar[r]&\Pic(\sY_{s,m+1})/H^{1}(\sY_{s},q^{*}\go m_{s}^{m}/q^{*}\go m_{s}^{m+1})\ar[r]&\Pic(\sY_{s,m})\ar[r]& H^{2}(\sY_{s},q^{*}\go m_{s}^{m}/q^{*}\go m_{s}^{m+1})
        }.
    \end{equation}\end{footnotesize}
     By the assumption $H^{1}(X,\sO_{X})=0$, we have $H^{1}(X_{s},\sO_{X_{s}})=0$, and then $$H^{1}(\sY_{s},q^{*}\go m_{s}^{m}/q^{*}\go m_{s}^{m+1})=0.$$ Thus the above diagram takes the following form:
     \begin{footnotesize}\begin{equation}
        \xymatrix@C=3ex{
        0\ar[r]&\Pic(X_{s,m+1})\ar[r]\ar[d]^{\alpha_{m+1}}&\Pic(X_{s,m})\ar[r]^{\delta_{X_{s,m}}\ \ \ \ \ \ }\ar[d]^{\alpha_{m}}& H^{2}(X_{s}, \sI_{s}^{m}/\sI_{s}^{m+1})\ar[d]\\
        0\ar[r]&\Pic(\sY_{s,m+1})\ar[r]&\Pic(\sY_{s,m})\ar[r]& H^{2}(\sY_{s},q^{*}\go m_{s}^{m}/q^{*}\go m_{s}^{m+1})
        }.
    \end{equation}\end{footnotesize}
     Notice that when $m=0$, $X_{s,m}\cong \sY_{s,m}$. We prove by induction that $$\Pic(X_{s,m})\cong \Pic(\sY_{s,m}).$$ Consider the commutative diagram
     \begin{footnotesize}\begin{equation}
        \xymatrix@C=3ex{
        0\ar[r]&\Pic(X_{s,m+1})\ar[r]\ar[d]^{\alpha_{m+1}}&\Pic(X_{s,m})\ar[r]^{\delta_{X_{s,m}}\ \ \ \ \ \ }\ar[d]^{\alpha_{m}}& \Im(\delta_{X_{s,m}}) \ar[r]\ar[d]^{\xi_m}&   0\\
        0\ar[r]&\Pic(\sY_{s,m+1})\ar[r]&\Pic(\sY_{s,m})\ar[r]& H^{2}(\sY_{s},q^{*}\go m_{s}^{m}/q^{*}\go m_{s}^{m+1})&
        },
    \end{equation}\end{footnotesize}
     where $\xi_m$ is the composition of the inclusion of $\Im(\delta_{X_{s,m}})\subset H^{2}(X_{s}, \sI_{s}^{m}/\sI_{s}^{m+1})$ with the map $H^{2}(X_{s}, \sI_{s}^{m}/\sI_{s}^{m+1})\rightarrow H^{2}(\sY_{s},q^{*}\go m^{m}_{s}/q^{*}\go m^{m+1}_{s})$, which is always injective by Proposition \ref{prop:inj for every step}. Since $\alpha_m$ is an isomorphism, and  $\xi_m$ is injective, then by the snake lemma, $\alpha_{m+1}$ is an isomorphism. As a result, we have $$\text{Image}(\Pic(X_{s,m})\rightarrow\Pic(X_{s}))\rightarrow\text{Image}(\Pic({\sY}_{s,m})\rightarrow\Pic(X_{s}))$$ is an isomorphism, i.e. we have $\text{INL}_{m}$ for all $m$.

     %$$\text{Image}(\Pic(X_{s,\ell+1})\rightarrow\Pic(X_{s,\ell}))\rightarrow\text{Image}(\Pic(\sY_{s,\ell+1})\rightarrow\Pic(\sY_{s,\ell}))$$ is surjective for all $\ell$. It holds when $\ell=0,1$ by the injectivity of $$H^{2}(X_{s}, \sI_{s}/\sI_{s}^{2})\rightarrow H^{2}(\sY_{s},q^{*}\go m_{s}/q^{*}\go m_{s}^{2}).$$ Now we assume it hold for all $\ell<m$. For an element $a\in \Pic(\sY_{s,m+1})$, by our induction there is an $b\in \Pic(X_{s,m})$ such that the image of $b$ in $\Pic(X_{s})$ is the same as $a$ in $\Pic(X_{s})$. By the injectivity of $\Pic(\sY_{s,m}\hookrightarrow\Pic(X_{s})$, we can see that $b$ has the same image as $a$ in $\Pic(\sY_{s,m})$ and it implies that the image of $b$ in $H^{2}(\sY_{s},q^{*}\go m_{s}^{m}/q^{*}\go m_{s}^{m+1})$ is $0$. Since $$H^{2}(X_{s}, \sI_{s}^{\ell}/\sI_{s}^{\ell+1})\rightarrow H^{2}(\sY_{s},q^{*}\go m_{s}^{\ell}/q^{*}\go m_{s}^{\ell+1})$$ is injective, then $b$ maps to $0$ in $H^{2}(X_{s}, \sI_{s}^{m}/\sI_{s}^{m+1})$, which implies that $b\in \Pic(X_{s,m+1})$ As a result, we have $$\text{Image}(\Pic(X_{s,m})\rightarrow\Pic(X_{s}))\rightarrow\text{Image}(\Pic(\hat{\sY}_{s,m})\rightarrow\Pic(X_{s}))$$ is an isomorphism, i.e. we have $\text{INL}_{m}$ for all $m$. 
    \end{proof}

\subsection{"Bigness" is Necessary}
	Now we assume $f:X\rightarrow C$ is a non-isotrivial elliptic surface with a section. We put
	\[
	\Sigma:=\{x\in C| f^{-1}(x) \;\text{is singular}\}.
	\]
	We assume for the moment that each singular fiber is irreducible (which implies that each fiber is nodal cubic or cuspidal cubic). 
	Then by Kodaira's canonical bundle formula, $\omega_{X}\simeq f^{*}(\omega_{C}\otimes\sL)$ where $\sL:=R^{1}f_{*}(\sO_{X})$ with $\deg\sL=\chi(\sO_{X})>0$. And we have the following lemma:
	\begin{lemma}For each $i$, 
		$f^{*}:H^{0}(C,\omega_{C}^{\otimes i})\rightarrow H^{0}(X,\omega_{X}^{\otimes i})$ is an isomorphism.
	\end{lemma}
	Then for a closed point $s\in\bm A$, we can associate a spectral surface $X_{s}$ over $X$ and a spectral curve $C_{s}$ over $C$. 
	\begin{lemma} The spectral curve $C_s$ and spectral surface $X_s$ are related in the following Cartesian diagram:
		\[\begin{tikzcd}
	X_s \arrow[r,"{\pi_s}"] \arrow[d] \arrow[dr, phantom, "\ulcorner", very near start]
	& X \arrow[d,"f"] \\
	C_s \arrow[r]
	& C
\end{tikzcd}.\]
		If $C_{s}$ is smooth and ramified outside of $\Sigma$, then $X_{s}$ is also smooth. In particular, generic spectral surfaces are smooth.
	\end{lemma}
	Notice that $\omega_{X}$ is nef but not big, and $\underline{\Pic}_{X_{s}}$ is not isomorphic to $\underline{\Pic}_{X}$.  Actually, we can calculate that $H^{0}(X_{s},\sO_{X_{s}})=\oplus_{i=0}^{r-1}H^{0}(X,\sL^{-i})$. Without the bigness, we can not prove the vanishing of $H^{0}(X,\sL^{-i})$ for $i>0$, thus $\dim \underline{\Pic}_{X_{s}}>\dim \underline{\Pic}_{X}$.
	
% 	We denote $\MW(S)$ the Mordell-Weil lattice of an elliptic surface $S$.
	\section{An Application to Hitchin System over Quintic Surfaces}
		In this section, we discuss generic fibers of Hitchin fibrations over a generic quintic surface of Picard number 1. We divide this section into two parts dealing with bundle case and torsion free sheaf case. The main result is that if we consider the moduli of torsion free Higgs sheaves, then generic fibers of corresponding Hitchin maps are Hilbert schemes of points of spectral surfaces. And we calculate the number of connected components. 
		
		Now, let $X\subset \PP^3$ be a smooth quintic surface with Picard number 1 and the base field $k$ will be an algebraically closed field with odd or zero characteristic.
\vspace{10pt}
\paragraph*{\textbf{Basic properties of $X$ and $X_s$}} Let us investigate some basic properties of $X$. 
\blemma $X$ admits a lift on $W_2(k)$.\elemma
\begin{proof}
% By adjunction formula, the canonical bundle $\omega_X$ equals to $\sO_X(1)=\sO_{\PP^3}(1)|_X$.
By the exact sequence of structure sheaves 
	\[0\to\sO_{\PP^3}(-4)\to\sO_{\PP^3}(1)\to \sO_{X}(1)\to0,\]
we have	$h^0(X,\sO_X(1))=4$ and $h^1(X,\sO_X(1))=0$. Let us consider the Euler sequence
%	\[0\to\sO_{\PP^3}\to\sO_{\PP^3}(-1)^{\oplus 4}\to \Omega_{\PP^3}^1\to0\]
	\[0\to i^*\Omega_{\PP^3}^1(1)\to\sO_{X}^{\oplus 4}\xrightarrow{ev} \sO_{X}(1)\to0,\]
which shows $h^0(X,i^*\Omega_{\PP^3}^1(1))=0$. Since $\omega_{X}=\sO_{X}(1)$, we have
	\[	0\rightarrow \sO_{X}(-4)\rightarrow i^{*}\Omega^{1}_{\BP^{3}}(1)\rightarrow \Omega^{1}_{X}\otimes \omega_X\rightarrow 0.\]
Because $h^1(X,\sO_X(-4))=0$, $H^2(X,\sT_X)\cong H^0(X,\Omega_X^1\otimes \omega_X)^\vee=0$. Thus $X$ admits a lifting to $W_{2}(k)$.
\end{proof}

%$H^2(X,\sT_X)\cong H^0(X,\Omega^1_X\otimes \omega_X)$, 
By definition of Todd class, $\Td(X)=1+\frac{1}{2}c_{1}(TX)+\frac{1}{12}(c_{1}^{2}(TX)+c_{2}(TX))$, the following computation shows
	
	\begin{lemma}\label{lem:todd class of quintic}
		The Todd class of $X$ is $\Td(X)=1-\frac{1}{2}c_{1}(\omega_{X})+c_{1}(\omega_{X})^{2}$.
	\end{lemma}
	\begin{proof}
		%	By the Euler exact sequence:
		%	\[
		%	0\rightarrow\Omega^{1}_{\BP^{3}}\rightarrow \sO(-1)^{\oplus 4}\rightarrow \sO_{\BP^{3}}\rightarrow 0
		%	\]	
		%	we have $\Ch(\Omega^{1}_{\BP^{3}})=4\Ch(\sO(-1))-1$, thus $c_{1}(\Omega^{1}_{\BP^{3}})=-4c_{1}(\sO(1))$ and $c_{2}(\Omega^{1}_{\BP^{3}})=6c_{1}^{2}(\sO(1))$. 
		
		Let $\sI_{X}$ be the ideal sheaf defining $X$ in $\BP^{3}$ and $i:X\rightarrow \BP^{3}$ is the closed immersion, then
		\[
		0\rightarrow \sI_{X}/\sI_{X}^{2}\rightarrow i^{*}\Omega^{1}_{\BP^{3}}\rightarrow \Omega^{1}_{X}\rightarrow 0.
		\]
		Notice that $\sI_{X}/\sI_{X}^{2}$ is $i^{*}\sO(-5)$, thus \[\begin{split}
		    \tx{ch}(\Omega^{1}_{X})=&i^{*}\tx{ch}(\Omega^{1}_{\BP^{3}})-i^{*}\tx{ch}(\sO(-5))\\
		  %  \tx{ by Euler sequence }
		    =&4i^*\tx{ch}(\sO(-1))-i^*\tx{ch}(\sO)-i^*\tx{ch}(\sO(-5))\\
		  % \tx{ ch is a ring map }
		   =&2+\bf h+\frac{-21}{2}\bf h^2\\
		  %  \tx{ definition of ch}
		    =&2+c_1(\omega_X)+\frac{1}{2}(c_1(\omega_X)^2-2c_2(\Omega_X^1)),
		\end{split}\]
		where $\bf h=c_1(i^*\sO(1))=c_1(\omega_X)$.
		%Since $c_{1}(\Omega^{1}_{\BP^{3}})=-4c_{1}(\sO(1))$ and $c_{2}(\Omega^{1}_{\BP^{3}})=6c_{1}^{2}(\sO(1))$, 
    		We have $c_{1}(\Omega^1_{X})=\bf h$ and $c_{2}(\Omega^{1}_{X})=11\bf h^{2}$. Thus $\Td(X)=1+\frac{1}{2}c_{1}(TX)+\frac{1}{12}(c_{1}^{2}(TX)+c_{2}(TX))=1-\frac{1}{2}\bf h+\bf h^{2}$.
	\end{proof}
	
	We take $\sL$ as $\omega_X=\sO_X(1)$ and $\bm A$ as Hitchin base in previous sections. Moreover, we have the open subset $U\subset \bm A$ parametrizing the smooth spectral surfaces. Let $s\in U$ and $X_s$ be the corresponding spectral curve. 
	\begin{lemma}\label{lem:canonical line bundle of spectral}
		The canonical line bundle $\omega_{X_{s}}$ of $X_{s}$ is isomorphic to $\pi_{s}^{*}(\omega_{X}^{\otimes r})$.
	\end{lemma}
 	\begin{proof}
 	    By Lemma \ref{lem:can bund formula}, $\omega_{Y}\cong \sO_{Y/X}(-2)$, and thus $\omega_{Y}|_{X_{s}}$ it trivial. By the following exact sequence:
 	    \[
 	    0\rightarrow (\pi^{*}\omega_{X}^{-r}\otimes\sO_{Y/X}(-r))|_{X_{s}}\rightarrow\Omega_{Y}^{1}|_{X_{s}}\rightarrow\Omega_{X_{s}}^{1}\rightarrow 0,
 	    \]
 	    we have $\omega_{X_{s}}\cong \pi_{s}^{*}(\omega_{X}^{\otimes r})$.
 	\end{proof}
 We then compute the Todd class of $X_s$. By the exact sequences
 	\[0\to \Omega_{Y/X}^1\to \pi^*(\omega_X^\vee\oplus\sO_X)(-1)\to \sO_{Y}\to 0, \]
 	\[0\to \pi^*\Omega_X^1\to \Omega_{Y}^1\to \Omega_{Y/X}^1\to 0,\]
 	\[0\to \pi^*\omega_X^{\otimes -r}(-r)|_{X_s}\to \Omega_{Y}^1|_{X_s}\to \Omega_{X_s}^1\to 0,\]
 we have
 	\[\begin{split}\tx{ch}(\Omega_{X_s}^1)=&i_s^*\tx{ch}(\Omega_{\bf P}^1)-\tx{ch}(\pi^*_s(\omega_X)^{\otimes (-r)})\\
 	=&\pi^*_s\tx{ch}(\Omega_X^1)+i_s^*\tx{ch}(\Omega_{\bf P/X}^1)-\tx{ch}(\omega_{X_s}^\vee)\\
 	=&\pi^*_s\tx{ch}(\Omega_X^1)+\pi^*_s\tx{ch}(\omega_X^\vee)-\tx{ch}(\omega_{X_s}^\vee)\\
 %	=&\pi_s^*(2+c_1(\omega_X)-\frac{21}{2}c_1(\omega_X)^2+1-c_1(\omega_X)+\frac{1}{2}c_1(\omega_X)^2-1+r\cdot c_1(\omega_X)-\frac{r^2}{2}c_1(\omega_X)^2)\\
 	=&\pi_s^*(2+r\cdot c_1(\omega_X)-\frac{20+r^2}{2}\cdot c_1(\omega_X)^2)\\
 	=&\pi_s^*[2+r\cdot c_1(\omega_X)+\frac{1}{2}\cdot(r^2c_1(\omega_X)^2-2(10+r^2)\cdot c_1(\omega_X)^2)].
 	\end{split}\]
We have $c_{1}(\Omega_{X_s}^1)=r\cdot \pi_s^*\bf h$ and $c_{2}(\Omega^{1}_{X_s})=(10+r^2)\cdot \pi_s^*\bf h^{2}$. Thus $\Td(X_s)=1+\frac{1}{2}c_{1}(TX_s)+\frac{1}{12}(c_{1}^{2}(TX_s)+c_{2}(TX_s))=1-\frac{r}{2}\cdot \pi_s^*\bf h+\frac{(r^2+5)}{6}\cdot\pi^*_s\bf h^2$.
\vspace{20pt}

\paragraph*{\textbf{Concepts of Higgs sheaves}} Let us introduce Higgs bundles (sheaves) and Hitchin maps in the sense of Tanaka and Thomas \cite{TT20}.
	
	\begin{definition}
	    A Higgs sheaf on $X$ is such a pair $(\sE,\theta)$ with $\sE$ is coherent sheaf on $X$ and $\theta:\sE\rightarrow \sE\otimes \omega_{X}$ is an $\sO_{X}$ linear morphism. We say $(\sE,\theta)$ is a Higgs bundle (torsion free Higgs sheaf) if $\sE$ is a vector bundle (resp. torsion free sheaf ) on $X$. 
	\end{definition}
	\begin{definition}
	    We denote $\tx{Higgs}_{r,c_{1},c_{2}}^o$ (resp. $\Higgs_{r,c_{1},c_{2}}$) the moduli stack of Higgs bundles (resp. torsion free Higgs sheaves)  with $\rk \sE=r, c_{1}(\sE)=c_{1}, c_{2}(\sE)=c_{2}$. We call the affine space $\bm A:=\oplus_{i=1}^{r}H^{0}(X,\omega_{X}^{\otimes i})$ the Hitchin base. The following characteristic polynomial map:
	    \[
	    h:\Higgs_{r,c_{1},c_{2}}\rightarrow\bm A,\ (\sE,\theta)\mapsto \tx{char.poly.}(\theta)
	    \]
	    is called the Hitchin map. We put $h|_{\tx{Higgs}_{r,c_{1},c_{2}}^o}=h^{o}$.
	\end{definition}

	\paragraph*{\textbf{Moduli of Higgs sheaves on $X$}} In what follows, we discuss the moduli of Higgs sheaves on our quintic surface $X$.
\subsection{Bundle Case}
	Given a Higgs bundle $(\sE,\theta)$, with image $s\in \bm A$. We can see that $(\sE,\theta)$ can be treated as a coherent sheaf on $X_{s}$. We assume $s$ is generic, thus $X_{s}$ is smooth. Let $\pi_{s}:X_{s}\rightarrow X$ be the projection, which is finite flat of degree $r=\rk \sE$. 
	\begin{proposition}\label{prop:bundle over spectral}
		Let $\sM$ be a coherent sheaf over $X_{s}$, then ${\pi_{s}}_*\sM$ is locally free sheaf of rank $r$ if and only if $\sM$ is an invertible sheaf on $X_s$. 
	\end{proposition}
	\begin{proof}
		By \cite[Section 2]{BBG97}, ${\pi_{s}}_*\sM$ is locally free of rank $r$ over $X$ if and only if $\sM$ is a maximal Cohen-Macaulay module over $X_s$ of generic rank 1. Then by \cite[Corollary 3.9]{BD08}, over a normal surface, a coherent sheaf is Cohen-Macaulay if and only if it is reflexive. Thus $\sM$ is reflexive of generic rank 1. Since $X_{s}$ is regular, $\sM$ is a line bundle.
	\end{proof}

Let us fix $(r,c_{1},c_{2})$ as before. We assume ${\pi_{s}}_*\sM=\sE$, where $\sM$ is an invertible sheaf over $X_{s}$. By Grothendieck-Riemann-Roch, we have:
	\begin{align}
	{\pi_{s}}_*(\tx{ch}(\sM)\cdot\Td(X_{s}))=\tx{ch}(\sE)\cdot\Td(X),\\
	{\pi_{s}}_*(\Td(X_{s}))=\tx{ch}({\pi_{s}}_*\sO_{X_{s}})\cdot\Td(X).
	\end{align}
As $\sM$ is locally free of rank 1, by definition of the Chern character, we have $\tx{ch}(\sM)=1+c_{1}(\sM)+\frac{c_{1}(\sM)^{2}}{2}$. From the above two equations, we have:
	\begin{equation}\label{eq: GRR}
	{\pi_{s}}_*((c_{1}(\sM)+\frac{c_{1}(\sM)^{2}}{2})\cdot\Td(X_{s}))=(\tx{ch}(\sE)-\tx{ch}({\pi_{s}}_*\sO_{X_{s}}))\cdot\Td(X).
	\end{equation}
By our previous computation of Todd classes of $X$ and $X_s$, the left hand side of (\ref{eq: GRR}) becomes
 \[{\pi_{s}}_*((c_{1}(\sM)+\frac{1}{2}c_{1}(\sM)^{2})\cdot\Td(X_{s}))= \pi_{s*}c_{1}(\sM)+\pi_{s*}\frac{1}{2}c_{1}(\sM)^{2}-\frac{r}{2}{\pi_{s}}_*(c_{1}(\sM))\cdot \bf h,\]
 and the right hand side of (\ref{eq: GRR}) becomes
 \[\begin{split}(\tx{ch}(\sE)-\tx{ch}({\pi_{s}}_*\sO_{X_{s}}))\cdot\Td(X)=&(\tx{ch}(\sE)-\tx{ch}({\pi_{s}}_*\sO_{X_{s}}))\cdot (1-\frac{1}{2}\cdot \bf h+\bf h^{2})\\
 =&c_1(\sE)+\frac{r(r-1)}{2}\bf h
 + \tx{ch}_2(\sE)\\&-\frac{r(r-1)(2r-1)}{12}\bf h^2-\frac{1}{2}c_1(\sE)\cdot \bf h\\&-\frac{r(r-1)}{4}\bf h^2\\
 =&c_1(\sE)+\frac{r(r-1)}{2}\bf h+\tx{ch}_2(\sE)-\frac{1}{2}c_1(\sE)\cdot\bf h\\&-\frac{r(r-1)(r+1)}{6}\bf h^2.
 \end{split}\]
 	To conclude:

% 	\begin{proof}
% 		By Lemma \ref{lem:canonical line bundle of spectral}, 
% 		LHS of \eqref{eq: GRR}$={\pi_{s}}_*(c_{1}(\sL)+\frac{1}{2}c_{1}(\sL)^{2})-\frac{r}{2}{\pi_{s}}*(c_{1}(\sL))\cdot c_{1}(\omega_{X})
% 		$
		
% 		By Lemma \ref{lem:todd class of quintic},
% 		RHS of \eqref{eq: GRR}$=((c_{1}(E)+\frac{r(r-1)}{2}c_{1}(K_{X}))+(\Ch_{2}(E)-\frac{r(r-1)(2r-1)}{12}c_{1}^{2}(K_{X}))\cdot (1-\frac{c_{1}(\omega_{X})}{2}+c_{1}(\omega_{X})^{2})$.
		
% 		Thus, we have the following equations:
% 		\begin{align*}
% 		&\pi_{s*}(c_{1}(\sL))=c_{1}(E)+\frac{r(r-1)}{2}c_{1}(K_{X});\\
% 		&-\frac{r}{2}\pi_{s*}(c_{1}(\sL))\cdot c_{1}(K_{X})+\pi_{s*}(\frac{c_{1}^{2}(\sL)}{2})\\
% 		&=-\frac{1}{2}(c_{1}(E)+\frac{r(r-1)}{2}c_{1}(K_{X}))\cdot c_{1}(K_{X})+\Ch_{2}(E)-\frac{r(r-1)(2r-1)}{12}c_{1}^{2}(K_{X})
% 		\end{align*}
% 		which gives the equations in the proposition.
% 	\end{proof}
	\begin{lemma}\label{lem:top equation}
	     Denote $c_1(\omega_X)$ by $\bf h$, the Chern classes of $\sM$ has to satisfy the following equations:
 		\begin{align}
 		\pi_{s*}(c_{1}(\sM))=&c_{1}(\sE)+\frac{r(r-1)}{2}\cdot\bf h,\\
 		\pi_{s*}\frac{1}{2}c_{1}(\sM)^{2}-\frac{r}{2}{\pi_{s}}_*(c_{1}(\sM))\cdot\bf h=&\tx{ch}_2(\sE)-\frac{1}{2}c_1(\sE)\cdot\bf h-\frac{r(r-1)(r+1)}{6}\bf h^2. 
 		%\\
 		%&-\frac{r}{2}\pi_{s*}(c_{1}(\sL))\cdot c_{1}(\omega_{X})+\pi_{s*}(\frac{c_{1}^{2}(\sL)}{2})\\
 		%&=-\frac{1}{2}c_{1}(\sE)\cdot c_{1}(\omega_{X})-\frac{r(r-1)(r+1)}{6}c_{1}^{2}(\omega_{X})+\frac{c^{2}_{1}(E)-2c_{2}(E)}{2}
 		\end{align}
	\end{lemma}
	Now we can compute generic fibers of $h^{o}$.
	\begin{theorem}
	   	If the generic fiber of the Hitchin map $h^{o}$ is not empty, then it is a single point, i.e., the mapping degree of $h^{o}$ is 1.
	\end{theorem}
    \begin{proof}

 		As we have proved that $\pi_{s}^{*}:\Pic(X)\rightarrow\Pic(X_{s})$ is an isomorphism for very general $s\in\bm V$. Now we fix $s\in \bm V$, then $\Pic(X_s)= \Z\cdot [\pi_s^*\omega_X]\cong \Z$, and then we can assume $\sM\cong \pi_{s}^{*}(\omega_{X}^{\otimes \mu})$ for some integer $\mu$. Thus we have:
	\[
	r\mu \bf h=c_{1}(\sE)+\frac{r(r-1)}{2}\bf h
	\]
	and 
	\[
	\frac{r\mu^{2}-r^{2}\mu}{2}{\bf h}^{2}=\tx{ch}_2(\sE)-\frac{1}{2}(c_{1}(\sE)+\frac{r(r-1)}{2}\bf h)\bf h-\frac{r(r-1)(2r-1)}{12}{\bf h}^2,
	\]
	combined together with equations in Lemma \ref{lem:top equation}, we have:
	\begin{align*}
	    c_{1}(\sE)=&(r\mu-\frac{1}{2}r(r-1))\bf h,\\
	    \tx{ch}_{2}(\sE)=&[\frac{1}{2}(r\mu+r\mu^{2}-r^{2}\mu)+\frac{r(r-1)(2r-1)}{12}]{\bf h}^2.
	\end{align*}
	Thus $\mu$ is uniquely determined by $c_{1}(\sE)$ and $c_{2}(\sE)$. The theorem then follows.
    \end{proof}
	
\subsection{Torsion Free Sheaves Case}
	
	Now let us suppose $(\sE,\theta)$ is a Higgs sheaf over $X$ with image $s$ in $\bm A$. Similarly, we have $(\sE,\theta)$ can be identified with torsion free sheaf of rank $1$ on $X_{s}$. We may denoted it by $\sM\otimes \sI_{\Delta}$. Here $\sI_{\Delta}$ is the ideal sheaf of a closed subscheme of finite length $\ell_{\Delta}$ on $X_{s}$. 
	\begin{lemma}
		$\tx{ch}(\sM\otimes \sI_{\Delta})=\tx{ch}(\sM)+i_{*}[\Delta]$.
	\end{lemma}
	\begin{proof}
		We denote by $i:\Delta\rightarrow X$ the closed immersion. Again by Grothendieck-Riemann-Roch, we have $c_{1}(i_{*}\sO_{\Delta})=0,c_{2}(i_{*}\sO_{\Delta})=-i_{*}[\Delta]$.
	\end{proof}
	
	\begin{theorem}\label{thm:generic fibers are Hilbert schemes}
	   Generic fibers of $h$ are Hilbert schemes of points of spectral surfaces (if nonempty) and generic fibers are connected.
	\end{theorem}
    \begin{proof}
    Now still by Grothendieck-Riemann-Roch, the equations in Lemma \ref{lem:top equation} become:
	\begin{align*}
	\pi_{s*}(c_{1}(\sM))=&c_{1}(\sE)+\frac{r(r-1)}{2}\bf h,\\
	   -\frac{r}{2}\pi_{s*}(c_{1}(\sM))\bf h+\pi_{s*}(\frac{c_{1}^{2}(\sM)}{2}+[\Delta])=&-\frac{1}{2}c_{1}(\sE)\bf h-\frac{r(r-1)(r+1)}{6} {\bf h}^2+\tx{ch}_2(\sE).
	\end{align*}

	Similarly, we take $s$ very general and assume $\sM\cong \pi_{s}^{*}(\omega_{X}^{\otimes \mu})$ for some integer $\mu$, and thus, 
	\begin{align*}
	    c_{1}(\sE)=&(r\mu-\frac{1}{2}r(r-1))\bf h,\\
	    \tx{ch}_{2}(\sE)=&[\frac{1}{2}(r\mu+r\mu^{2}-r^{2}\mu)+\frac{r(r-1)(2r-1)}{12}]{\bf h}^2+\pi_{s*}([\Delta]).
	\end{align*}
	Thus $\mu$ and $[\Delta]$ is uniquely determined by $c_{1}(\sE)$ and $c_{2}(\sE)$. In other words, the generic fibers are connected if nonempty.
	\end{proof}
\newcommand{\etalchar}[1]{$^{#1}$}

\end{document}